\documentclass[final,3p,times]{elsarticle}
\usepackage{amssymb}
\usepackage[usenames, dvipsnames]{color}
\usepackage{amsmath,amscd,amssymb,latexsym}
\usepackage{latexsym}
\usepackage{amsthm}
\usepackage{amsmath}
\newtheorem{thm}{Theorem}[section]
\newtheorem{rem}{Remark}[section]
\newtheorem{lem}{Lemma}[section]
\newtheorem{pro}{Proposition}[section]

\newtheorem{definition}{Definition}

\newcommand{\ud}{\mathrm{d}}

\allowdisplaybreaks
\numberwithin{equation}{section}
\newtheorem{thmlet}{Theorem}

\journal{}

\begin{document}           

\begin{frontmatter}
\title{Multiplicity results for fractional magnetic problems involving exponential growth}
\author[Pawan Kumar Mishra]{Pawan Kumar Mishra}
\ead{pawanmishra@mat.ufpb.br}
\author[J. M. do O]{Jo\~ao Marcos do \'O\corref{mycorrespondingauthor}}
\cortext[mycorrespondingauthor]{Corresponding author}
\ead{jmbo@pq.cnpq.br}
\author[Manasses de Souza]{Manass\'es de Souza}
\ead{manassesxavier@hotmail.com}
\address[ Manasses de Souza, J. M. do O, Pawan Kumar Mishra]{Department of Mathematics, Federal University of  Para\'iba,\\
Jo\~ao Pessoa, PB, 58051--900, Brazil}

\begin{abstract}
\noindent We study the following fractional elliptic equations of the type,
\begin{equation*}
(-\Delta)^{\frac12}_A u = \lambda u+f(|u|)u ,\;\textrm{in } \;(-1, 1),\; u=0\;\textrm{in } \;\mathbb R\setminus (-1, 1),
\end{equation*}
where  $\lambda$ is a positive real parameter and $(-\Delta)^{\frac12}_A$ is the fractional magnetic operator with $A:\mathbb R\to \mathbb R$ being a smooth  magnetic field. Using a classical critical point theorems, we prove the existence of multiple solutions in the non-resonant case when the nonlinear term $f(t)$ has a critical exponential growth in the sense of Trudinger-Moser inequality. 
\end{abstract}


\end{frontmatter}
\section{Introduction}
\noindent We study the following fractional elliptic equations of the type,
\begin{equation*} 
(P)_\lambda \left\{
\begin{array}{rl}
(-\Delta)^{\frac12}_A u &= \lambda u+f(|u|)u ,\;\textrm{in } \;\Omega,\\
 u&=0\;\quad \quad \quad \; \quad\textrm{in } \;\mathbb R\setminus \Omega,
\end{array}
\right.
\end{equation*}
where $\Omega = (-1,1)$ and $\lambda$ is a positive real parameter. For a magnetic field $A: \mathbb R\to \mathbb R$  the operator $(-\Delta)^{\frac12}_A$ is known as the fractional magnetic operator. This operator, recently introduced in \cite{AS}, has been defined (upto a normalization constant) as follows
\begin{equation}\label{FMO}
 -(-\Delta)_A^{\frac12}u(x)= \frac{1}{\pi}\lim_{\epsilon \to 0}\int_{\mathbb{R}\setminus B(x, \epsilon)} \frac{u(x)-e^{i(x-y)\cdot A(\frac{x+y}{2})}u(y)}{|x-y|^{2}}\ud y,\;\;\;\;\; \; x\in \mathbb{R},
 \end{equation}
 where $B(x, \epsilon)$ denotes the real interval of size  $\epsilon$ around $x$. It is clear that, when $A=0$, the above operator is consistent with the usual fractional Laplacian operator (square root of Laplacian) which has seized a lot of attention in the recent past, see \cite{FL1, FL2, FL3, FL4} and references therein. This operator arises in the description of various phenomena in the several branches of applied sciences, for example, \cite{App1} uses the fractional Laplacian for linear and nonlinear lossy media, \cite{App2, Applbm} use the fractional Laplacian for option pricing in jump diffusion and exponential L\'evy models,  \cite{App3} provides the first ever derivation of the fractional Laplacian operator as a means to represent the mean friction in the turbulence modeling and many more. 

On the other hand,  we can interpret  $(-\Delta)_{A}^{\frac12}$
as a fractional analog of the magnetic
Laplacian $(\nabla - \mathrm{i} A)^2$, with $A$ being a
{bounded} potential. In particular, for the physical interest,  the study of \eqref{FMO} is apparent in the case $s={1}/{2}$. Indeed
the operator in \eqref{FMO} takes inspiration from the definition of a quantized operator corresponding to the classical relativistic Hamiltonian symbol for a relativistic particle of mass $m\geq 0$, that is
$$
\sqrt{(\xi-A(x))^{2}+m^{2}}+V(x), \quad (\xi, x)\in \mathbb R^{N}\times \mathbb R^{N},
$$
which is the sum of the kinetic energy term involving $A(x)$ (magnetic vector potential) and $V(x)$ (potential energy term of electric scalar potential). For the sake of completeness, we emphasized that in the literature there are three kinds of quantum relativistic Hamiltonians depending on how to quantize the kinetic energy term $\sqrt{(\xi-A(x))^{2}+m^{2}}$. As explained in \cite{T2}, these three nonlocal operators are in general different from each other but coincide when the vector potential $A$ is assumed to be linear, so in particular, in the case of constant magnetic fields.
For a more detailed description of the operator  $(-\Delta)^{s}_{A}$ and relaed problems, we refer the interested readers to \cite{AS, FiscellaV1, T1, T2,  AME1, AME, BMX} and the references therein.

In a latest work, authors in \cite{FiscellaVecchi} studied a multiplicity result for the following problem in higher dimensions  involving fractional magnetic operator
\begin{equation}\label{FPr}
(-\Delta )_A^su=\lambda u+|u
|^{2^*_s-2}u, \;\; \textrm{in}\;\;\Omega,\;\; u=0\;\;\textrm{on}\;\;\mathbb R^N\setminus\Omega,
\end{equation}
where $\Omega\subset \mathbb R^N$ is an open and bounded set with Lipschitz boundary, $N>2s$, $s\in(0, 1)$ and $2^*_s=2N/(N-2s)$ is the fractional critical Sobolev exponent. The result summarizes as the existence of  $m$ pairs of solution of \eqref{FPr} for $\lambda$ lying in the suitable left neighborhood of any eigenvalue with multiplicity $m$ of the magnetic fractional Laplace operator with Dirichlet boundary data. 

Note that when $A=0$, the problem \eqref{FPr} transforms into the following problem involving the celebrated fractional Laplace operator
\begin{equation*}
(-\Delta )^su=\lambda u+u^{2^*_s-2}u, \;\; \textrm{in}\;\;\Omega,\;\; u=0\;\;\textrm{on}\;\;\mathbb R^N\setminus \Omega,
\end{equation*}
which has been studied in fair share by the authors in \cite{fbm}. Using the abstract critical point theorem, authors in \cite{fbm} have generalized the results of Cerami, Fortuno and Struwe \cite{sfs} for the nonlocal setting.

 We know from classical fractional Sobolev embedding that $H^{s}(\Omega)$ is continuously embedded in $L^q(\Omega)$ for all $q\in [1, 2^*_s]$, where $2^*_s=2N/(N-2s)$. Note that formally, $2^*_s=\infty$ if $N=2s$. Since $s\in (0, 1)$, the only choice for this fact to be true is $N=1$ and $s=1/2$. At this point a natural question arises to look for an optimal space where $H^{1/2}(\Omega)$ can be embedded.  This answer was first given by Ozawa \cite{Ozawa} and later improved by  Iula,  Maalaoui and  Martinazzi \cite{matrinazziint} in the form of fractional Trudinger-Moser inequlity (see Lemma \ref{Trudinger-Moser1}).  This result has motivated many researcher to consider the critical exponent problem in limiting case of fractional Sobolev embedding in dimension 1 such as \cite{MR3399183, JPS, IzSq}, with no attempt to give a complete list.  (see also in the local case \cite{Adimurthi, FMR} and reference therein)

 The result obtained in \cite{FiscellaVecchi} covers all the dimensions except the dimensions $N=2s$ which corresponds to the only dimension, $N=1$ when $s\in (0, 1)$ for $s=1/2$. Up to the best of our knowledge, there is no work dealing with fractional magnetic operator with critical exponential growth except the work of Ambrosio \cite{Ambrossio} which is related with concentration behavior of solutions for nonlinear Schr\"odinger equations. This is the main motivation for studying the problem under consideration. In the case of $A=0$, a non-magnetic counter part of $(P)_\lambda$ was partially considered in \cite{EJDE2}.  Inspired from a suitable variant of Trudinger-Moser inequality and by proving  the required Moser sequence estimates to study the min-max level under Adimurthi type assumption (see assumption $(H_4)$ below), we have complemented the work of \cite{FiscellaVecchi} in the dimension one. Our results complete the partial 
 result obtained in \cite{EJDE2} as well.

\subsection{Assumptions} 
We will consider $f: [0,+\infty) \rightarrow [0,+\infty)$ a continuous function having the critical exponential growth in the following sense: there exists $\beta_0 >0$ such that
\begin{equation}\label{CG}
\tag{CG}
\lim_{t \to +\infty} \frac{f(t) t}{e^{\beta t^2}} = \left\{
\begin{array}{ll}
+\infty, & \hbox{if}\,\, \beta < \beta_0 \\
0, & \hbox{if}\,\, \beta > \beta_0.
\end{array}
\right.
\end{equation}
Moreover, $f$ satisfies
\begin{description}
	\item[$(H_1)$]\label{H1} There exist $t_0 >0$ and $M>0$ such that
	\[F(t) \leq M f(t) t,\,\,\mbox{for all}\,\, t \geq t_0;\]
	where
	\[
	F(t) := \int_0^t f(\tau) \tau \ud \tau,\,\, \mbox{for any}\,\, t>0.
	\]
	\item[$(H_2)$] $0< F(t) \leq f(t) t^2$ for all $t>0$.
	\item[$(H_3)$]  For each $k\geq 1$, $\lambda$ $f$ satisfies $\ell:=\limsup_{t\to 0^+} \frac{2F(t)}{t^2}<\lambda_k-\lambda$, where $\lambda\in (0, \lambda_1)$ if $k=1$ otherwise $\lambda \in (\lambda_{k-1}, \lambda_k)$.
\end{description}

Here $(\lambda_k)$ denotes the sequence of eigenvalues associated to the problem
\begin{align}\label{EVP}
(-\Delta)_A^\frac{1}{2}u&=\lambda u\;\;\textrm{in}\;\;\Omega,\; u=0\;\;\;\; \textrm{in}\;\;\mathbb R\setminus \Omega.
\end{align}
\subsection{Spectral properties}
It is known that there exists a infinite sequence of eigenvalues $\lambda_1\leq\lambda_2\leq .....\leq\lambda_k\leq ....$ with $\lambda_k\rightarrow+\infty$ as $k\rightarrow \infty$. The eigenfunctions $\{\varphi_k\}$ corresponding to each eigenvalue $\lambda_k$ form an orthonormal basis for $L^2(\Omega)$ and an orthogonal basis for $X_{0, A}$, where the space $X_{0, A}$ and corresponding norm $\|\cdot\|_{X_{0, A}}$ is defined in Section 2.
Hence $X_{0,A} = H_{k} \oplus H_{k}^{\perp}$, where $H_{k} = \mbox{span}_{\mathbb{R}}\{\phi_1, \phi_2,\cdot\cdot\cdot,\phi_{k}\}$.
The following characterization is shown in Proposition 3.3 \cite{FiscellaVecchi}.
\begin{equation*}\label{EVL1}
\lambda_1=\min_{u\in X_{0, A}\setminus \{0\}}\frac{\|u\|^2_{X_{0, A}}}{\|u\|_2^2}
\end{equation*}
Moreover, inductively, for any $k\geq 2$
\begin{equation}\label{EVL2}
\lambda_{k}=\min_{u\in H^\perp_k\setminus \{0\}}\frac{\|u\|^2_{X_{0, A}}}{\|u\|_2^2}
\end{equation}

\subsection{Main results and remarks}

The objective of this paper is multi-fold.  Depending on the location of the parameter $\lambda$ with respect to the spectrum of $(-\Delta)_A^\frac{1}{2}$ with Dirichlet data, we categories the result of this paper in the form of following four main Theorems. The first result deals with the case when the parameter $0<\lambda<\lambda_1$ and the nonlinearity has critical exponential growth. Note that  the problem under consideration  is no more  coercive which is a natural hindrance to study via usual minimization argument. In this case the classical mountain pass theorem gives the existence of a critical point of the corresponding energy functional which results into a nontrivial weak solution of the problem by a one to one correspondence between critical points of the associated energy functional and weak solutions of the problem. Our first result is stated as follows:
\begin{thm}\label{thm1} Assume $(H_1)-(H_2)$ and $(H_3)$ with $k=1$. Let  $f$ has exponential critical growth together with 
\begin{description}
	\item[$(H_4)$] $\displaystyle{\liminf_{t \to +\infty} f(t)t^2 e^{\beta_0 t^2}=+\infty}$, {where} $\beta_0$ is introduced in \eqref{CG}.
	\end{description}
	Then the problem $(P)_\lambda$ has a nontrivial solution.
\end{thm}
\begin{rem}
	We point out that the assumption $(H_4)$ was introduced by Adimurthi in \cite{Adimurthi} in the first instance.. 
	\end{rem}

The problem $(P)_\lambda$ exhibits interesting feature when the parameter $\lambda$ lies in between the eigen values $\lambda\in ( \lambda_k, \lambda_{k+1})$ for $k\geq 1$.  The second result of the paper highlights this delicate point. The proof of this result invokes the celebrated idea of Linking geomtery.
\begin{thm}\label{thm2}
 Assume $\lambda\in (\lambda_k, \lambda_{k+1})$, $(H_1)-(H_4)$ and that $f$ has exponential critical growth. 
	Then problem $(P)_\lambda$ has a nontrivial solution.
\end{thm}

The third theorem of the paper also deals with the critical growth nonlinearity but involves a little stronger assumption due to D. M. Cao (see assumption $(H_5)$ below) instead of Adimurthi assumption $(H_4)$. But with this compromise, we could prove the least bound of critical points of the associated functional by applying another abstract critical point theorem due to \cite{crlit}. The result says that
\begin{thm}\label{thm3}
 Assume $(H_1)-(H_3)$ and that $f$ has exponential critical growth. Furthermore assume
	\begin{description}
		\item[$(H_5)$] there exist $p>2$  and a constant $C_p>0$ possibly large such that $f(t)t \geq C_p t^{p-1}$ for all $t\geq 0$.
	\end{description}
	Define $\lambda_k$ be the $k^{th}$ eigenvalue of the problem \eqref{EVP} with multiplicity $m$. Let $\lambda\in \mathbb R$ and define $\lambda^*=\min\{\lambda_k: \lambda<\lambda_k\}$. If $\lambda<\lambda^*$ and
	\[
	C_p>\left( \frac{\beta_0(p-2)}{\pi}\right)^\frac{p-2}{2}\left((\lambda^*-\lambda)2^{\frac{p}{p-2}}\right)^\frac{p}{2},
	\]
	where $\beta_0$ is introduced in \eqref{CG}, then problem $(P)_\lambda$ admits $m$ pairs of non-trivial weak solutions $\{-u_{\lambda,j}, u_{\lambda, j}\}$, for every $j=1,2,\cdot\cdot\cdot,m$.
\end{thm}

\begin{rem}
	The assumption $(H_5)$ was firstly introduced by D. M. Cao in \cite{Cao}.
\end{rem}
Before stating the last result of the paper, we introduce what we mean by subcritical growth. We say that $f$ has subcritical growth at $+\infty$ if 
\begin{equation}\label{SG}
\tag{SG}
\lim_{t\to+\infty}\frac{|f(t)t|}{e^{\beta t^2}}=0\;\; \text{for all} \;\;\beta>0.
\end{equation}
Under the light of \eqref{SG}, it is clear that for some constant $C(\beta)>0$ the nonlinearity satisfies
\begin{equation}\label{subgrowth}
f(t)t\leq C(\beta)e^{\beta t^2}\;\;\text{for all}\;\;\beta>0\;\;\text{ and for all}\;\; t\in \mathbb R.
\end{equation}
In the last Theorem of the paper, we show that the problem $(P)_\lambda$ exhibits two non-trivial weak solutions under subcritical growth assumption in the sence of \eqref{SG}. In this case by allowing the nonlinearity to be subcritical, we could prove our result without assuming Cao condition $(H_5)$ or Adimurthi type assumption $(H_4)$.

We conclude the buildup of the last result by introducing the following notations. Since the space $X_{0, A}\hookrightarrow L^{p}(\Omega)$ for all $p\in [2,\; \infty) $, the following supremum  is well defined

\begin{equation}\label{Holder}
S_p=\sup_{\{v\in X_{0, A}\;:\; \|v\|_{X_{0, A}}\leq 1\}} \frac{\|v\|_p}{\|v\|_{X_{0, A}}}.
\end{equation}

\begin{thm}\label{thm4}
	Assume $f$ satisfies $(H_1)-(H_2)$ together with \eqref{SG}.
	Then for every $\rho>(\sqrt{2}S_pC(\beta))^2$ there exists
	\[
	\Lambda(\rho):=\frac{1}{S_2^2}\left(1-\frac{\sqrt{2}S_pC(\beta)}{\sqrt{\rho}}\right),
	\]
	where $S_2$, $S_p$   and $C(\beta)$ are defined in \eqref{Holder} and \eqref{subgrowth}, respectively,
	such that problem $(P)_\lambda$ has at least two nontrivial weak solutions for every $\lambda \in(0, \Lambda(\rho))$, one of which has norm strictly less than $\rho$.
\end{thm}
The proof of the above Theroem  is variational and is based on a abstract critical point theorem due to Recceri \cite{Ricceri} (see Theorem 6).
\begin{rem}
	We point out  that these results are true even in the absence of magnetic field, that is, the case when $A=0$.
\end{rem}

\section{Functional framework}
In this section we give a more general variational set up rather than considering $\Omega=(-1, 1), N=1, s=1/2$ as in our case in this paper. We indicate with $|\Omega|$ the $N$-dimensional Lebesgue
measure of a measurable set $\Omega \subset \mathbb{R}^N$. Moreover,
for every $z \in \mathbb{C}$ we denote by $\Re z$ its real part, and by
$\overline{z}$ its complex conjugate.
Let $\Omega \subset \mathbb{R}^N$ be an open set.
We denote by $L^2(\Omega,\mathbb{C})$
the space of measurable functions $u:\Omega\to\mathbb{C}$ such that
$$
\|u\|_{L^2(\Omega)}=\left(\int_{\Omega}|u(x)|^2 \, \ud x\right)^{1/2}<\infty,
$$
where $|\cdot|$ is the Euclidean norm in $\mathbb{C}$.

 For  $s\in (0,1)$, we define the magnetic Gagliardo semi-norm as
$$
[u]_{H^{s}_A(\Omega)}:=\left(\frac{1}{2\pi}\iint_{\Omega\times\Omega}\frac{|u(x)-e^{\mathrm{i} (x-y)\cdot
A(\frac{x+y}{2})}u(y)|^2}{|x-y|^{N+2s}}\,\ud x\,\ud y\right)^{1/2}.
$$
We denote by $H^{s}_A(\Omega)$ the space of functions
$u\in L^2(\Omega,\mathbb{C})$ such that $[u]_{H^{s}_A(\Omega)}<\infty$, normed with
$$
\|u\|_{H^{s}_A(\Omega)}:=\left(\|u\|_{L^2(\Omega)}^2+[u]_{H^{s}_A(\Omega)}^2\right)^{1/2}.
$$

However, to encode the boundary condition $u=0$ in $\mathbb{R}^N\setminus\Omega$,
the natural functional space
to deal with weak solutions of problem $(P)_\lambda$ is
$$
X_{0,A} := \left\{ u \in H^{s}_{A}(\mathbb{R}^N): u = 0\ \text{in }
\mathbb{R}^N \setminus \Omega \right\}.
$$

We define the following real scalar product on $X_{0,A}$
$$
\langle u,v\rangle_{X_{0,A}}:=\frac{1}{2\pi} \Re \iint_{\mathbb{R}^{2N}}\frac{\big(u(x)-e^{\mathrm{i} (x-y)
\cdot A(\frac{x+y}{2})}u(y)\big)
\overline{\big(v(x)-e^{\mathrm{i} (x-y)\cdot A(\frac{x+y}{2})}v(y)\big)}}
{|x-y|^{N+2s}}\, \ud x \ud y,
$$
which induces the  norm
$$
\|u\|_{X_{0,A}} := \left( \frac{1}{2\pi}\iint_{\mathbb{R}^{2N}}\frac{|u(x)- e^{\mathrm{i} (x-y)\cdot
 A(\frac{x+y}{2})}u(y)|^2}{|x-y|^{N+2s}}\, \ud x \ud y \right)^{1/2}.
$$
Under the scalar product defined above, the space
 $(X_{0,A}, \langle \cdot,\cdot\rangle_{X_{0,A}})$ is a Hilbert space and
hence reflexive.

\medskip

Arguing similar to \cite{AA} and \cite{AS}, we have the following result.
\begin{lem}
\begin{enumerate}
\item [(i)]The space $H^{1/2}_A(\mathbb R, \mathbb C)$ is continuously embedded into $L^r(\mathbb R, \mathbb C)$ for any $r\in [2,\infty)$
and compactly embedded into $L^r_{\mathrm {loc}}(\mathbb R, \mathbb C)$ for any $r\in [1,\infty)$.
\item[(ii)] For any $u\in H^{1/2}_A(\mathbb R, \mathbb C)$, we get $|u|\in H^{1/2}(\mathbb R, \mathbb R)$ and $[|u|]\leq  [u]_A$. Moreover we also have the following pointwise diamagnetic inequality
\[
\big||u(x)|-|u(y)|\big|\leq \left|u(x)-e^{i(x-y)\cdot A(\frac{x+y}{2})}u(y)\right| \;\; a. ~e. \;\;x, y \in\mathbb R.
\]
\item[(iii)] If $u\in H^{1/2}(\mathbb R, \mathbb R)$ and has compact support then $v=e^{iA(0)\cdot x} u\in H^{1/2}_A(\mathbb R, \mathbb C)$.
\end{enumerate}
\end{lem}

As discussed in the introduction, the problems of the type $(P)_\lambda$ are motivated by the following version of the Trudinger-Moser inequality, which is a consequence of the results proved by Ozawa \cite{Ozawa},
Kozono, Sato and Wadade \cite{Wadade}, Martinazzi \cite{Martinazzi} and Takahashi \cite{Takahasi}.
	
	\begin{lem}\label{Trudinger-Moser1} If $\alpha >0$ and $u\in X_{0,0}$, it holds
		\begin{equation*}\label{TM11}
		e^{\alpha u^2} \in L^1(\Omega).
		\end{equation*}
		Moreover,
		\begin{equation*}
		\sup_{\{u \in X_{0,0}\,:\, \|u\|_{1/2,2} \leq 1\}}\int_\Omega e^{\alpha u^{2}} \, \ud x < \infty,
		\label{TM21}
		\end{equation*}
		for all $0 \leq \alpha  \leq \pi$, where
		\[
		\|u\|_{1/2,2} := \left(\dfrac{1}{2\pi}\displaystyle\int_{\mathbb R^2}
		\dfrac{(u(x)-u(y))^{2}}{|x-y|^{2}} \, \ud x\, \ud y  \right)^{1/2}.
		\]
	\end{lem}

	\begin{lem}\label{lem}
		If $0\leq \alpha  \leq \pi$, it holds
		\begin{equation}\label{Corolario1}
		\sup_{\{u \in X_{0,A}\,:\, \|u\|_{X_{0,A}} \leq 1\}}\int_\Omega e^{\alpha |u|^{2}} \, \ud x < \infty.
		\end{equation}
		Moreover, for any $\alpha >0$ and $u \in X_{0,A}$,
		\begin{equation}\label{lema2.3}
		e^{\alpha |u|^{2}} \in L^{1}(\Omega).
		\end{equation}
	\end{lem}
	
	\begin{proof} The estimating \eqref{Corolario1} follows from $\||u|\|_{1/2,2} \leq \|u\|_{0,A}$ and Lemma \ref{Trudinger-Moser1}.
		Now we prove the second part of the lemma. Given $u \in X_{0,A}$ and
		$\varepsilon>0$, there exists $\varphi \in C_{0}^{\infty}(\Omega)$
		such that $\||u|-\varphi\|_{1/2,2}<\varepsilon$. Since
		\begin{equation*}
		e^{\alpha |u|^{2}} \leq
		e^{\alpha(2(|u|-\varphi)^{2}+2\varphi^{2})} \leq
		\dfrac{1}{2}e^{4\alpha
			(|u|-\varphi)^{2}}+\dfrac{1}{2}e^{4\alpha
			\varphi^{2}},
		\end{equation*}
		it follows that
		\begin{equation}\label{jjjj1}
		\int_\Omega e^{\alpha |u|^{2}} \, \ud x \leq
		\dfrac{1}{2}\int_\Omega e^{4\alpha\||u|-\varphi\|_{1/2,2}^{2}\left(\frac{|u|-\varphi}{\||u|-\varphi\|_{1/2,2}}\right)^{2}}
		\, \ud x + \dfrac{1}{2}\int_\Omega e^{4\alpha\varphi^{2}} \,
		\ud x.
		\end{equation}
		Choosing $\varepsilon >0$ such that ${4\alpha\varepsilon^{2}} <
		\pi$, we have ${4\alpha\||u|-\varphi\|_{1/2,2}^{2}}<
		\pi$. Then, from Lemma \ref{Trudinger-Moser1} and
		\eqref{jjjj1}, we obtain
		\begin{equation*}
		\int_\Omega e^{\alpha |u|^{2}} \, \ud x \leq \dfrac{C}{2} +
		\dfrac{1}{2} \int_{\textrm{supp}(\varphi)}e^{4\alpha \varphi^{2}} \,
		\ud x< \infty.
		\end{equation*}
		Thus, the proof is complete.
	\end{proof}

	\begin{definition}
		We say that a function $u \in X_{0,A}$
		is a weak solution of $(P)_\lambda$ if  
		$$
		\langle u,\varphi \rangle_{X_{0,A}}= \lambda \Re \int_{\Omega}u(x)~ \overline{\varphi(x)}\, \ud x
		+ \Re \int_{\Omega}f(|u|)~u~ \overline{\varphi(x)}\, \ud x,
		$$
		for every $\varphi \in X_{0,A}$.
	\end{definition}
	Clearly, the weak solutions of $(P)_\lambda$ are the critical points of
	the Euler--Lagrange functional $\mathcal{I}_{A,\lambda}:X_{0,A} \to \mathbb{R}$,
	associated with $(P)_\lambda$, that is
	\begin{equation}\label{Jlam}
	\mathcal{I}_{A,\lambda}(u):= \frac12\|u\|_{X_{0,A}}^2
	- \frac{\lambda}{2}\|u\|_{2}^2
	- \int_\Omega F(|u|) \ud x,
	\end{equation}
	where $F(t)=\int_0^t f(s)s\ud s$. By using our assumptions and Lemma \ref{lem}, it is easy to see that $\mathcal{I}_{A,\lambda}$ is well-defined and of
	class $C^1(X_{0,A}, \mathbb{R})$.

	The next lemma will be used to ensure the geometry of the
	functional $\mathcal{I}_{A,\lambda}$.
	
	\begin{lem}
		If $v \in X_{0,A}$, $\alpha >0$, $q>2$ and $\|v\|_{X_{0,A}}\leq M$ with $\alpha
		M^{2}<\pi$, then there exists $C=C(\alpha,M,q)>0$
		such that
		\begin{equation*}
		\int_\Omega e^{\alpha |v|^{2}}|v|^{q}\ud x \leq C\|v\|_{X_{0,A}}^{q}.
		\end{equation*}
	\end{lem}
	
	\begin{proof}
		Taking $r > 1$ close to $1$ such that $\alpha r M^{2}<\pi $ and $r'q \geq 1$, where $r'={r}/{(r-1)}$. By H\"{o}lder's inequality, we have
		\begin{equation*}\label{exponential11}
		\int_\Omega e^{\alpha |v|^{2}}|v|^{q}\ud x \leq
		\left(\int_\Omega e^{\alpha r |v|^{2}}\ud
		x\right)^{1/r}\|v\|_{r'q}^{q} = \left(\int_\Omega e^{\alpha r \|v\|_{X_{0,A}}^{2}
			\left(\frac{|v|}{\|v\|_{X_{0,A}}}\right)^{2}}\,\ud x \right)^{1/r}
		\|v\|_{r'q}^q.
		\end{equation*}
		Since $\alpha r
		M^{2}<\pi$, it follows from \eqref{Corolario1} and the
		continuous embedding $X_{0,A} \hookrightarrow L^{r'q}(\Omega)$, that
		\[
		\int_\Omega e^{\alpha |v|^{2}}|v|^{q}\,\ud x \leq C\|v\|_{X_{0,A}}^q.
		\]
		Thus, the proof is complete.
	\end{proof}
	
	\bigskip
	
	We will show a refinement of
	\eqref{Corolario1}. This result will be crucial to show that the
	functional $\mathcal{I}_{A,\lambda}$ satisfies the Palais-Smale condition.

	\begin{lem}\label{tipolions}(P. L. Lions' concentration compactness result)
		If $(v_{n})$ is a sequence in $X_{0,A}$ with $\|v_{n}\|_{X_{0,A}}=1$ for all $n
		\in \mathbb{N}$ and $v_{n}\rightharpoonup v$ weakly in $X_{0,A}$,
		$0<\|v\|_{X_{0,A}}<1$, then for all
		$0<t<\pi(1-\|v\|_{X_{0,A}}^{2})^{-1}$, we have
		\begin{equation*}
		\sup_{n} \int_\Omega e^{t|v_{n}|^{2}}\,\ud x<\infty.
		\end{equation*}
	\end{lem}
	\begin{proof}
		Since $v_{n}\rightharpoonup v$ weakly in $X_{0,A}$ and $\|v_{n}\|_{X_{0,A}}=1$, we get
		\begin{equation*}
		\|v_{n}-v\|_{X_{0,A}}^{2}=1-2\langle v_{n},v \rangle_{X_{0,A}}+\|v\|_{X_{0,A}}^{2}\rightarrow
		1-\|v\|_{X_{0,A}}^{2}<\dfrac{\pi}{t}.
		\end{equation*}
		Thus, for $n \in \mathbb{N}$ enough large, we have
		$t\|v_{n}-v\|_{X_{0,A}}^{2}<\pi$. Thus, we may choose $q>1$
		close to 1 and $\varepsilon>0$ satisfying
		\begin{equation}\label{ddddjj12}
		qt(1+\varepsilon^{2})\|v_{n}-v\|_{X_{0,A}}^{2}<\pi,
		\end{equation}
		for $n \in \mathbb{N}$ enough large. By \eqref{Corolario1} and
		\eqref{ddddjj12}, there exists $C>0$ such that
		\begin{equation}\label{26}
		\int_\Omega e^{qt(1+\varepsilon^{2})|v_{n}-v|^{2}} \, \ud x=
		\int_\Omega e^{qt(1+\varepsilon)^{2}\|v_{n}-v\|_{X_{0,A}}^{2}\left(\frac{|v_{n}-v|}{\|v_{n}-v\|_{X_{0,A}}}
			\right)^{2}}
		\, \ud x \leq C.
		\end{equation}
		Moreover, since
		\begin{equation*}
		t|v_{n}|^{2}\leq
		t(1+\varepsilon^{2})|v_{n}-v|^{2}+t\left(1+\dfrac{1}{\varepsilon^{2}}\right)|v|^{2},
		\end{equation*}
		it follows by convexity of the exponential function with
		$q^{-1}+r^{-1}=1$ that
		\begin{eqnarray*}
			e^{t|v_{n}|^{2}}\leq
			\dfrac{1}{q}e^{qt(1+\varepsilon^{2})|v_{n}-v|^{2}}+\dfrac{1}{r}e^{rt(1+1/\varepsilon^{2})|v|^{2}}.
		\end{eqnarray*}
		Therefore, by \eqref{lema2.3} and \eqref{26}, we get
		\begin{equation*}
		\int_\Omega e^{t|v_{n}|^{2}}\, \ud x \leq
		\dfrac{1}{q}\int_\Omega e^{qt(1+\varepsilon^{2})|v_{n}-v|^{2}} \,
		\ud x + \int_\Omega e^{rt(1+1/\varepsilon^{2})|v|^{2}}\, \ud x \leq
		C,
		\end{equation*}
		and the result is proved.
	\end{proof}
	
\section{Palais-Smale sequence analysis}

In this section, we will study the definition and properties of Palais-Smale sequence and its precompactness. We begin by recalling the following definition of Palais-Smale sequence.

\begin{definition}
$\{u_n\}\subset X_{0, A}$ is called a Palais-Smale sequence for $\mathcal I_{A, \lambda}$ at a level $c$ (in short $(PS)_c$ sequence) if
\[
\mathcal I_{A,\lambda}(u_n)\to c\;\;\text{ and } \mathcal I_{A, \lambda}^\prime (u_n)\to 0\;\;\text{ as}\;\; n\to \infty.
\]
We say that $\mathcal I_{A,\lambda}$ satisfies Palais-Smale condition at level $c$ if any $(PS)_c$ sequence admits a convergent subsequence in $X_{0, A}$.
\end{definition}

\begin{lem}\label{PS11}
Assume $(H_1)$.
Let $(u_n) \subset X_{0,A}$ be a $(PS)_c$ sequence of $I_\lambda$. Then $(u_n)$ is a bounded in $X_{0,A}$.
\end{lem}

\begin{proof}
Let $(u_n) \subset X_{0,A}$ be a $(PS)_c$ sequence of $I_\lambda$, that is,
\begin{equation}\label{EJ23}
\frac{1}{2}\|u_n\|_{X_{0,A}}^2 - \frac{\lambda}{2}\|u_n\|_2^2 - \int_{\Omega} F(|u_n|) \ud x \to c,\,\, \mbox{as}\,\, n \to +\infty,
\end{equation}
and
\begin{equation}\label{EJ24}
\left|\Re\langle u_n, v \rangle_{X_{0,A}} - \lambda \Re\langle u_n , v\rangle_{L^2} - \Re\int_{\Omega} f(|u_n|) u_n \overline{v} \ud x \right| \leq \varepsilon_n \|v\|_{X_{0,A}},\,\, \mbox{for all}\,\,
v \in X_{0,A},
\end{equation}
where $\varepsilon_n \to 0$ as $n \to +\infty$. It follows from $(H_1)$, that there exists $t_1 >0$ such that
\begin{equation}\label{EJ231}
F(t) \leq  \frac{1}{4}f(t) t^2,\,\,\mbox{for all}\,\, t \geq t_1.
\end{equation}
Using \eqref{EJ23} - \eqref{EJ231} and $v=u_n$ as test function, we can find $C>0$ such that
\begin{equation}\label{EJ26}
\int_{\Omega} f(|u_n|) |u_n|^2  \ud x \leq C + 2 \varepsilon_n \|u_n\|_{X_{0,A}}.
\end{equation}
To complete the proof, we consider two cases.

\medskip

\noindent\textbf{Case 1:} $0< \lambda < \lambda_1$

\medskip

\noindent From \eqref{EJ24}, \eqref{EJ26} and the variational characterization of $\lambda_1$, we have the following estimate for $\|u_n\|_{X_{0,A}}$,
\[
\left(\frac{\lambda_1- \lambda}{\lambda_1} \right)\|u_n\|_{X_{0,A}}^2 \leq C + 3 \varepsilon_n \|u_n\|_{X_{0,A}}.
\]
Consequently, $(u_n)$ is a bounded sequence in $X_{0,A}$ in this case.

\medskip

\noindent\textbf{Case 2:} $\lambda_k < \lambda < \lambda_{k+1}$

\medskip

\noindent
Given $u \in X_{0,A}$, we write $u = u^{k} + u^{\perp}$, where $u^{k} \in H_{k}$ and $u^{\perp} \in H_{k}^{\perp}$. Notice that
\begin{equation}\label{EJ210}
\Re\langle u, u^{k}\rangle_{X_{0,A}} - \lambda \Re\langle u, u^{k}\rangle_{L^2} = \|u^{k}\|_{X_{0,A}}^2 - \lambda \|u^{k}\|_2^2
\end{equation}
and
\begin{equation}\label{EJ211}
\Re\langle u, u^{\perp}\rangle_{X_{0,A}} - \lambda \Re\langle u, u^{\perp}\rangle_{L^2} = \|u^{\perp}\|_{X_{0,A}}^2 - \lambda \|u^{\perp}\|_2^2.
\end{equation}
By \eqref{EJ24}, \eqref{EJ210} and the variational characterization of $\lambda_{k}$, we obtain
\[
\begin{aligned}
 -\varepsilon_n \|u_n^{k}\|_{X_{0,A}} & \leq \Re\langle u_n, u_n^{k}\rangle_{X_{0,A}} - \lambda \Re\langle u_n, u_n^{k}\rangle_{L^2} -
 \Re\int_{\Omega} f(|u_n|) u_n \overline{u_n^{k}} \ud x \\
& \leq \left(\frac{\lambda_{k}- \lambda}{\lambda_{k}} \right)\|u_n^{k}\|_{X_{0,A}}^2 - \Re\int_{\Omega} f(|u_n|) u_n \overline{u_n^{k}}\, \ud x.
\end{aligned}
\]
Therefore, we can find $C>0$ such that
\begin{equation}\label{EJ212}
\left(\frac{\lambda- \lambda_k}{\lambda_{k}} \right)\|u_n^{k}\|_{X_{0,A}}^2 \leq \varepsilon_n \|u_n^{k}\|_{X_{0,A}} + \|u_n^{k}\|_{\infty}  \int_{\Omega} f(|u_n|) |u_n| \,\ud x.
\end{equation}
By applying \eqref{EJ26}, there exist $C_1, C_2>0$ such that
\begin{equation}\label{EJ27}
\int_{\Omega} f(|u_n|) |u_n|  \,\ud x \leq C_1 + C_2\varepsilon_n \|u_n\|_{X_{0,A}}.
\end{equation}
Since $H_{k}$ is a finite dimensional subspace, we can find $C>0$ such that
$\|u_n^{k}\|_{\infty} \leq C \|u_n^{k}\|_{X_{0,A}}$. Thus,
from \eqref{EJ212} and \eqref{EJ27}, we get
\begin{equation}\label{EJ213}
\|u_n^{k}\|^2_{X_{0,A}} \leq C (\|u_n^{k}\|_{X_{0,A}} +  \varepsilon_n \|u_n^{k}\|_{X_{0,A}} + \varepsilon_n \|u_n^{k}\| _{X_{0,A}}\|u_n\|_{X_{0,A}}).
\end{equation}
Again, by \eqref{EJ24}, \eqref{EJ211} and the variational characterization of $\lambda_{k+1}$, it follows that
\[
\begin{aligned}
 \varepsilon_n \|u_n^{\perp}\|_{X_{0,A}} & \geq \Re\langle u_n, u_n^{\perp}\rangle_{X_{0,A}} - \lambda \Re\langle u_n, u_n^{\perp}\rangle_{L^2} -
 \Re\int_{\Omega} f(|u_n|) u_n \overline{u_n^{\perp}} \,\ud x \\
& \geq \left(\frac{\lambda_{k+1}- \lambda}{\lambda_{k+1}} \right)\|u_n^{\perp}\|^2_{X_{0,A}} - \Re\int_{\Omega} f(|u_n|) u_n \overline{u_n^{\perp}}\, \ud x\\
& \geq \left(\frac{\lambda_{k+1}- \lambda}{\lambda_{k+1}} \right)\|u_n^{\perp}\|^2_{X_{0,A}} - \int_{\Omega} f(|u_n|) |u_n|^2\, \ud x - \|u_n^{k}\|_{\infty}  \int_{\Omega} f(|u_n|) |u_n| \,\ud x.
\end{aligned}
\]
This together with \eqref{EJ26} and \eqref{EJ27}, implies that there exists $C > 0$ such that
\begin{equation}\label{EJ214}
\|u_n^{\perp}\|^2_{X_{0,A}} \leq C (\varepsilon_n \|u_n^{\perp}\|_{X_{0,A}}+ C + \varepsilon_n \|u_n\| _{X_{0,A}}+  \|u_n^{k}\| _{X_{0,A}}+ \varepsilon_n \|u_n^{k}\|_{X_{0,A}} \|u_n\|_{X_{0,A}}).
\end{equation}
Combining \eqref{EJ213} and \eqref{EJ214}, we obtain that
\begin{equation*}\label{EJ215}
\|u_n\|^2_{X_{0,A}}\leq C (1+ \|u_n\| _{X_{0,A}}+  \varepsilon_n \|u_n\|^2_{X_{0,A}}),
\end{equation*}
and consequently the sequence $\{u_n\}$ is bounded. Thus, we finished the proof.
\end{proof}

\begin{lem}\label{PS}
Assume that $(H_1)-(H_2)$ are satisfied.
Then $\mathcal{I}_{A,\lambda}$ satisfies the $(PS)_c$ condition for $c< \frac{\pi}{2 \beta_0}$.
\end{lem}

\begin{proof}
Let $(u_n)$ be satisfying \eqref{EJ23} and \eqref{EJ24}. By Lemma \ref{PS11}, we obtain a subsequence denoted again by $(u_n)$ such that, for some $u \in X_{0,A}$,
we have $u_n \rightharpoonup u$ in $X_{0,A}$, $u_n \rightarrow u$ in $L^q(\Omega)$ for all $q \in [1, +\infty)$ and $u_n(x) \rightarrow u(x)$
a.e in $\Omega$. It follows from \eqref{EJ26} and \cite[Lemma 2.1]{FMR}, that $\int_{\Omega} f(|u_n|)|u_n|\ud x \rightarrow \int_{\Omega} f(|u|)|u|\ud x$, as $n \to +\infty$.
Thus, by applying $(H_1)$ and the Generalized Lebesgue Dominated Convergence Theorem we have
\[
\int_{\Omega} F(u_n)\ud x \rightarrow \int_{\Omega} F(u)\ud x,\,\, \mbox{as}\,\, n \to +\infty.
\]
This convergence together with \eqref{EJ23} imply
\begin{equation}\label{EJ221}
\lim_{n \to +\infty}\|u_n\|^2_{X_{0,A}} = 2c + \lambda\|u\|_2^2 +2 \int_{\Omega} F(|u|) \ud x.
\end{equation}
Consequently, from \eqref{EJ24} it follows
\begin{equation}\label{EJ222}
\lim_{n \to +\infty} \int_{\Omega} f(|u_n|)|u_n|^2 \ud x= 2c  +2 \int_{\Omega} F(|u|) \ud x.
\end{equation}
From $(H_2)$ and \eqref{EJ222} we reach $c \geq 0$. It follows by $(H_2)$ and \eqref{EJ24}, that
\[
\|u\|^2_{X_{0,A}} - \lambda \|u\|_2^2 = \int_{\Omega} f(|u|)|u|^2 \ud x \geq 2 \int_{\Omega} F(|u|) \ud x,
\]
and consequently, we get that $I_\lambda(u) \geq 0$.

\medskip

Now we will prove that $u_n \rightarrow u$ in $X_{0,A}$.

\medskip

In order to achieve this goal we will consider three cases.

\medskip

\noindent\textbf{Case 1:} $c=0$

\medskip

\noindent In this case, using \eqref{EJ221} we have
\[
0 \leq\mathcal{I}_{A,\lambda}(u) \leq \liminf_{n \to +\infty} \mathcal{I}_{A,\lambda}(u_n) = \liminf_{n \to +\infty} \frac{1}{2}\|u_n\|^2_{X_{0,A}} - \left(\frac{\lambda}{2}\|u\|_2^2 + \int_{\Omega} F(|u|) \ud x  \right)=0.
\]
Consequently, $u_n \rightarrow u$ in $X_{0,A}$, as $n \to +\infty$, as we wanted to demonstrate.
\medskip

\noindent\textbf{Case 2:} $c\not=0$ and $u=0$

\medskip

\noindent We will show that this case cannot happen for a $(PS)_c$ sequence.
Indeed, since $u=0$, it follows from \eqref{EJ221} that, given $\varepsilon >0$, for $n$ large enough, we have
\begin{equation}\label{EJ224}
\|u_n\|^2_{X_{0,A}}\leq 2 c + \varepsilon.
\end{equation}
Now we notice that, using that $f$ has critical growth, it holds
\begin{equation}\label{EJ225x}
\int_{\Omega} (f(|u_n|)|u_n|)^q \ud x \leq C \int_{\Omega} e^{q \beta \|u_n\|^2_{X_{0,A}} \left(\frac{|u_n|}{\|u_n\|_{X_{0,A}}}\right)^2} \ud x.
\end{equation}
Since $c < \frac{\pi}{2 \beta_0}$, by using \eqref{EJ224}, we can choose $q>1$ sufficiently close to $1$, $\beta > \beta_0$ sufficiently close to $\beta_0$,
and $\varepsilon$ sufficiently small such
that $q \beta \|u_n\|^2 _{X_{0,A}}< \pi$, for $n$ large enough. Thus, by the Trudinger-Moser inequality and \eqref{EJ225x} we have
\begin{equation}\label{ETC}
\int_{\Omega} (f(|u_n|)|u_n|)^q \ud x \leq C.
\end{equation}
From this estimate and the H\"{o}lder's inequality, up to a subsequence, we get
\[
\int_{\Omega} f(|u_n|)|u_n|^2 \ud x \rightarrow 0,\,\, \mbox{as}\,\, n \to +\infty.
\]
By using \eqref{EJ24}, we obtain $\|u_n\|_{X_{0,A}} \rightarrow 0$, as $n \to +\infty$. This contradicts \eqref{EJ221}.

\medskip

\noindent\textbf{Case 3:} $c\not=0$ and $u\not=0$

\medskip

\noindent Consider $v_n = \dfrac{u_n}{\|u_n\|_{X_{0,A}}}$ and $v= \dfrac{u}{\lim\|u_n\|_{X_{0,A}}}$.

\noindent It is clear that
$v_n \rightharpoonup v$ weakly in $X_{0,A}$. If $\|v\|_{X_{0,A}} =1$ we conclude the proof. Then, we assume that $\|v\| _{X_{0,A}}<1$.

\medskip

\noindent \textbf{Claim:} There exist $q>1$ sufficiently close to $1$ and $\beta > \beta_0$ sufficiently close to $\beta_0$,
such that
\[
q \beta \|u_n \|^2_{X_{0,A}}  < \frac{\pi}{1-\|v\|_{X_{0,A}}^2},
\]
for $n$ large enough. As a consequence this claim and Lemma \ref{tipolions} we have that \eqref{ETC} holds, and we can
see as in the \textbf{Case 2} that $u_n \rightarrow u$ strongly in $X_{0,A}$. So to complete the proof is enough to prove this statement.

\medskip

Notice that, up to a subsequence,
\begin{equation*}\label{EJ221x}
\lim_{n \to +\infty}\frac{1}{2}\|u_n\|^2 _{X_{0,A}}= c + \frac{\lambda}{2}\|u\|_2^2 + \int_{\Omega} F(|u|) \ud x.
\end{equation*}
Denote by
\[
B := \left(c + \frac{\lambda}{2}\|u\|_2^2 + \int_{\Omega} F(|u|) \ud x\right) (1-\|v\|^2_{X_{0,A}}).
\]
Then
\[B = c - \mathcal{I}_{A,\lambda}(u)\]
and consequently,
\[
\lim_{n \to +\infty}\frac{1}{2}\|u_n\|^2_{X_{0,A}}  = \frac{B}{1-\|v\|^2_{X_{0,A}}} =\frac{c- \mathcal{I}_{A,\lambda}(u)}{1-\|v\|^2_{X_{0,A}}} < \frac{\pi}{2 \beta_0(1-\|v\|^2_{X_{0,A}})}.
\]
This implies the claim.
\end{proof}

\section{Mountain pass case when $0<\lambda<\lambda_1$}

In this case we will use the Mountain Pass Theorem due to Ambrosetti and Rabinowitz \cite{AR1973}.

\begin{thmlet}
Let $J: H\to \mathbb R$ be a $C^1$ functional on a Banach space $(H, \; \|\cdot \|)$ satisfying
\begin{itemize}
\item[(i)] there exists some $\beta>0$ such that $J$ satisfies the Palais-Smale condition, $(PS)_c$ in short, for all $c\in (0, \beta)$,
\item[(ii)] there exist constants $\rho, \delta>0$ such that $J(u)\geq \delta$ for all $u\in H$ satisfying $\|u\|=\rho$.
\item[(iii)] $J(0)<\delta$ and $J(v)<\delta$ for some $v\in H\setminus \{0\}$ with $\|v\|\neq 0$.
\end{itemize}
Consider $\Gamma:=\{\eta\in C([0, 1],H): \eta(0)=0 \;\text{ and }\; \eta(1)=v\}$ and set
$
c_M=\inf_{\eta\in \Gamma}\max_{t\in [0, 1]} J(\eta(t))\geq \delta.
$
Then $c_M\in (0, \beta)$ and it is a critical point of the functional $J$.
\end{thmlet}

In the following propositions we will show the above geometry.

\begin{pro}\label{MPBelow1}
Assume that $f$ satisfies $(H_1)$. Then there exist $\tilde{M}>0$ and $u\in X_{0, A}$ such that $\mathcal I_{A, \lambda}(u)<-\tilde{M}$ for all $\lambda>0$.
\end{pro}
\begin{proof}
Let us fix some $u_0\in H_k$ with $u_0 \not=0$. Let us introduce the scalar map $\psi: \mathbb R\to \mathbb R$ defined as $\psi(t)=\mathcal I_{A, \lambda}(tu_0)$. Now from assumption $(H_1)$, there are $\mu>2$ and constants $C_1, C_2>0$ such that
\begin{equation}\label{implicationH1}
F(|u|)\geq C_1|u|^\mu-C_2 .
\end{equation}
Using \eqref{Jlam}, \eqref{implicationH1}  and equivalence of $ X_{0, A}$ and $L^\mu$ norms, we get
\[
\psi(t)\leq \frac{t^2}{2}\|u_0\|_{X_{0,A}}^2-C_1 t^\mu \|u_0\|_{X_{0,A}}^\mu +2C_2
\]
 which implies that $\psi(t)\to -\infty$ as $t\to \infty$. Hence the result follows.
\end{proof}

\begin{pro} \label{MPBelow}
Assume that $f$ satisfies $(H_3)$. Then there exist $\delta, \rho>0$ such that
$\mathcal I_{A,\lambda}(u)\geq \delta$ for $u\in X_{0, A}$ satisfying $\|u\|_{X_{0,A}}=\rho.$
\end{pro}
\begin{proof}
From assumption $(H_3)$, given $\epsilon>0$, there exists $\delta=\delta(\epsilon)>0$ such that
\[
\frac{2F(t)}{t^2}\leq \ell+\epsilon\;\;\text{for all}\;\; |t|\leq \delta.
\]
By the exponential critical growth assumption on the nonlinearity $f$, there exist $C=C(\epsilon)>0$ and $\beta>\beta_0$ such that
\[
F(|u|)\leq \frac{\ell+\epsilon}{2}|u|^2+C |u|^3 e^{\beta|u|^2},
\]
which implies
\[
\mathcal I_{A, \lambda} (u)\geq \frac{1}{2}\|u\|_{X_{0,A}}^2-\frac{\lambda}{2}\|u\|_2^2-\frac{\ell+\epsilon}{2}\|u\|_2^2-C\int_\Omega |u|^2e^{\beta_0|u|^2} \, \ud x.
\]
Therefore, for  $\|u\|_{X_{0,A}}=\rho>0$ sufficiently small such that $\beta \rho^2 < \pi$, by Trudinger-Moser inequality \eqref{Corolario1} and H\"{o}lder's inequality, we reach
\[
\mathcal I_{A, \lambda} (u)\geq \frac{1}{2}\|u\|_{X_{0,A}}^2-\frac{\lambda}{2}\|u\|_2^2-\frac{\ell+\epsilon}{2}\|u\|_2^2-C \|u\|^3_{X_{0,A}}=\frac{1}{2}\|u\|^2_{X_{0,A}}-\frac{\lambda+\ell+\epsilon}{2}\|u\|_2^2-C \|u\|^3_{X_{0,A}}
\]
possibly for different constant $C>0$. Now using the characterization of $\lambda_1$, we get
\[
\mathcal I_{A, \lambda} (u)\geq \frac{1}{2}\left(1-\frac{\lambda+\ell+\epsilon}{\lambda_1}\right)\|u\|^2_{X_{0,A}}-C \|u\|^3_{X_{0,A}}.
\]
Observe that for a given $\epsilon>0$,  sufficiently small, $\lambda+\ell+\epsilon<\lambda_1$. Hence
\[
\mathcal I_{A, \lambda} (u)\geq C_1\|u\|^2_{X_{0,A}}-C \|u\|^3_{X_{0,A}},
\]
where $C_1=\frac{1}{2}\left(1-\frac{\lambda+\ell+\epsilon}{\lambda_1}\right)>0$. Next we denote $g(t)=C_1\rho^2-C\rho^3$ and observe that $g(\rho)\to 0$ as $\rho\to 0$. Hence for sufficiently small $\rho>0$ there exists $\delta>0$ such that  $g(\rho)\geq \delta>0$. This completes the proof of the result.
\end{proof}

\subsection{The minimax level}

To show some estimates on the minimax level we require some facts on the Moser's functions defined by Moser \cite{moser} (see also \cite{Takahasi} for the fractional case). The Moser's functions are defined as follows
$$
M_n(x)=\left\{%
\begin{array}{ll}
\displaystyle    \sqrt{\log n}, & |x|< \frac{1}{n}, \\
\displaystyle    \frac{\log(1/|x|)}{\sqrt{\log n}}, & \frac{1}n\leq |x|< {1}, \\
0,& |x|\geq {1}. \\
\end{array}%
\right.
$$
The following proposition deals with the  asymptotic estimates on Moser's sequence.
\begin{lem} The following estimates are satisfied by $M_n$
	\begin{description}
		\item[(a)]
		$\|M_n\|_2^2=\frac{4}{\log n} + o_n(1)$;
		\item[(b)] $\|M_n\|^2_{X_{0,A}}\leq \pi+ O\left(\frac{1}{\log n}\right)$.
	\end{description}
\end{lem}
\begin{proof}
	The proof of item (a)  follows from Takahasi \cite{Takahasi} (see estimate in Equation (2.5)). To prove the item (b), we will use Euler's formula $(e^{i\theta}=\cos \theta+i \sin  \theta)$ with the notation $\xi_A= \mathrm{i} (x-y)\cdot
	A(\frac{x+y}{2})$ as below
	\begin{align*}
	\|M_n\|^2_{X_{0,A}}&= \iint_{\mathbb{R}^{2}}\frac{|M_n(x)- e^{\mathrm{i} (x-y)\cdot
			A(\frac{x+y}{2})}M_n(y)|^2}{|x-y|^{2}}\, \ud x \ud y \\
	&=\iint_{\mathbb{R}^{2}}\frac{|M_n(x)- \cos \xi_A M_n(y)-i \sin \xi_A M_n(y) |^2}{|x-y|^{2}}\, \ud x \ud y \\
	&=\iint_{\mathbb{R}^{2}}\frac{|M_n(x)- \cos \xi_A M_n(y)|^2+|\sin \xi_A M_n(y) |^2}{|x-y|^{2}}\, \ud x \ud y\\
	&=\iint_{\mathbb{R}^{2}}\frac{|M_n(x)- M_n(y)|^2+2(1-\cos \xi_A) M_n(x)M_n(y)}{|x-y|^{2}}\, \ud x \ud y.
	\end{align*}
	Hence
	\begin{equation}\label{1223}
	\|M_n\|^2_{X_{0,A}} = [M_n]^2_{H^s(\mathbb R)}+2\iint_{\mathbb{R}^{2}}\frac{(1-\cos \xi_A) M_n(x)M_n(y)}{|x-y|^{2}}\, \ud x \ud y.
	\end{equation}
	Let us estimate the second integral in the above equation as follows. Denote
	\[
	I = \int_{\mathbb{R}^2} \frac{2(1-\cos(\xi_A(x,y))M_n(x) M_n(y)}{|x-y|^2} \ud x \ud y =2( I_1 + I_2),
	\]
	where
	\[
	I_1 = \int_{|x-y| < \delta} \frac{(1-\cos(\xi_A(x,y))M_n(x) M_n(y)}{|x-y|^2} \ud x \ud y
	\]
	and
	\[
	I_2 = \int_{|x-y| \geq \delta} \frac{(1-\cos(\xi_A(x,y))M_n(x) M_n(y)}{|x-y|^2} \ud x \ud y,
	\]
	with $\delta >0$ to be chosen later.
	
	In order to estimate $I_1$, notice that since $A \in L^\infty(\mathbb{R})$, we have
	\[
	\lim_{|x-y|\rightarrow 0} \left|\frac{(1-\cos(\xi_A(x,y))}{|x-y|^2} \right|=
	\lim_{|x-y|\rightarrow 0}  \left|\frac{(1-\cos(\xi_A(x,y))}{\xi^2_A(x,y)} A^2\left(\frac{x+y}{2}\right) \right|=
	C_A.
	\]
	Then, given $\varepsilon >0$, there exists $\delta=\delta(\varepsilon)>0$ such that
	\[
	\left|\frac{(1-\cos(\xi_A(x,y))}{|x-y|^2} \right| \leq C_A + \varepsilon,\,\, \mbox{for all}\,\, |x-y| < \delta.
	\]
	From this we get that
	\begin{eqnarray*}
		I_1 &\leq & (C_A + \varepsilon )\int_{|x-y| < \delta} M_n(x) M_n(y) \ud x \ud y\\
		& \leq & \frac{1}{2}(C_A + \varepsilon )\int_{|x-y| < \delta} (M^2_n(x) + M^2_n(y)) \ud x \ud y\\
		& = & (C_A + \varepsilon )\int_{|x-y| < \delta} M^2_n(x)  \ud x \ud y\\
		& = & 2 \delta (C_A + \varepsilon ) \|M_n\|_2^2.
	\end{eqnarray*}
	For the integral $I_2$, notice that
	\begin{eqnarray*}
		I_2 &\leq & 2 \int_{|x-y| \geq \delta} \frac{M_n(x)  M_n(y)}{|x-y|^2} \ud x \ud y\\
		& \leq &  \int_{|x-y| \geq \delta} \frac{M^2_n(x) + M^2_n(y)}{|x-y|^2} \ud x \ud y\\
		& = &   2\int_{|x-y| \geq \delta} \frac{M^2_n(x) }{|x-y|^2} \ud x \ud y\\
		& = & \frac{4}{\delta} \|M_n\|_2^2.
	\end{eqnarray*}
	Combining $I_1$ and $I_2$, we obtain
	\[
	I \leq \left(4 \delta (C_A + \varepsilon )+ \frac{8}{\delta} \right) \|M_n\|_2^2.
	\]
	Using this last estimate, \eqref{1223} and (a), we reach (b).
	This completes the proof of the lemma.
\end{proof}

Now, we define the minimax level of $\mathcal I_{A, \lambda}$ by
\begin{equation*}\label{EL28}
c(n) : = \inf_{\gamma \in \Gamma} \max_{t \in [0,1]} \mathcal I_{A, \lambda} (\gamma(t)),
\end{equation*}
where
$$\Gamma = \{\gamma \in C([0,1], X_{0,A})\,:\, \gamma(0)=0\,\, \mbox{and}\,\, \gamma(1) = R_n z_n\},$$
$R_n$ being such that $\mathcal I_{A, \lambda} (R_n z_n) \leq 0$ and $z_n = \dfrac{M_n}{\|M_n\|_{X_{0,A}}}$.

\begin{pro}\label{levelminimax}
	Assume that $f$ satisfies $(H_4)$.
	Then there exists $n$ large enough such that $$c(n)< \dfrac{\pi}{2 \beta_0}.$$
\end{pro}

\begin{proof} It is sufficient to find $n\in \mathbb{N}$ such that
	\begin{equation}\label{EL49}
	\max_{t \geq 0} \mathcal I_{A, \lambda}(t z_n) < \frac{\pi}{2 \beta_0}.
	\end{equation}
	Suppose by contradiction that \eqref{EL49} does not hold. So, for all $n$, this maximum is larger than or equal to
	$\frac{\pi}{2 \beta_0}$ (it is indeed a maximum, in view of Proposition \ref{MPBelow1} and Proposition \ref{MPBelow}). Let $t_n >0$
	be such that
	\begin{equation}\label{EL410}
	\mathcal I_{A, \lambda}(t_n z_n) = \max_{t \geq 0} \mathcal I_{A, \lambda}(t z_n).
	\end{equation}
	Then, for all $n \in \mathbb{N}$,
	\[
\mathcal I_{A, \lambda}(t_n z_n) \geq \frac{\pi}{2 \beta_0},
	\]
	and, consequently, for all $n \in \mathbb{N}$, we have
	\begin{equation}\label{EL411}
	t_n^2 \geq \frac{\pi}{ \beta_0}.
	\end{equation}
	Let us prove that $t_n^2 \rightarrow \dfrac{\pi}{ \beta_0}$ as $n \to +\infty$. From \eqref{EL410}, we know that
	\[
	\frac{d}{dt}(\mathcal I_{A, \lambda}(t z_n))=0\,\,\mbox{when}\,\, t=t_n.
	\]
	Multiplying this last equation by $t_n$ and observing that $\|z_n\|_{X_{0,A}}=1$, we  have, for $n$ large enough, that
	\begin{equation}\label{EL413}
	t_n^2 \geq \int_{B_1(0)}  f(t_n z_n) t_n^2 z_n^2 \ud x.
	\end{equation}
	By $(H_4)$, it follows that given $\zeta >0$, there exists $t_{\zeta} >0$ such that
	\[
	f(t) t^2 \geq \zeta e^{\beta_0 t^2}\,\,\mbox{for all}\,\, t \geq t_\zeta.
	\]
	From \eqref{EL413}, for $n$ large enough, we obtain
	\begin{equation}\label{EL414}
	t_n^2 \geq \zeta \int_{B_{1/n}(0)}  e^{\beta_0 t_n^2 z_n^2} \ud x \geq 2\zeta  e^{\log n \left(\frac{\beta_0 t_n^2}{\pi + C\log n)^{-1}} -1 \right)},
	\end{equation}
	which implies that $\{t_n\}$ is bounded sequence. Moreover, \eqref{EL411} together with \eqref{EL414} gives us
	that $t_n^2 \searrow \dfrac{\pi}{ \beta_0}$ as $n \to +\infty$.

\medskip

In order to conclude the proof, observe that from \eqref{EL411} and \eqref{EL414}, we obtain
\[
t_n^2 \geq 2\zeta  e^{\log n \left(\frac{\beta_0 t_n^2 - \pi - C(\log n)^{-1}}{\pi + C(\log n)^{-1}} \right)} \geq 2 \zeta  e^{ \frac{ - C}{\pi + C(\log n)^{-1}}},
\]
which implies 
$$\zeta \leq \dfrac{\pi e^{\frac{C}{\pi}}}{2 \beta_0}.$$
Since $\zeta$ is arbitrarily large, we get a contradiction. 
This completes the proof of the proposition.
\end{proof}

\subsection{Proof of Theorem~\ref{thm1}}

To conclude Theorem \ref{thm1} we use Propositions \ref{MPBelow1}, \ref{MPBelow} and \ref{levelminimax},
and apply the Theorem A.

\section{Linking Case when $\lambda_k<\lambda<\lambda_{k+1}$ for \textcolor{red}{$k\geq 1$}}

In this case we use the following critical point theorem known as Linking theorem due to Ambrosetti and Rabinowitz \cite{AR1973}.

\begin{thmlet}
 Let $J: H\to \mathbb R$ be a $C^1$ functional on a Banach space $(H, \; \|\cdot \|)$ such that $H=H_1\oplus H_2$ with $\dim H_1<\infty$.  If $J$ satisfies the following
\begin{itemize}
\item [(i)] there exists some $\beta>0$ such that $J$ satisfies the Palais-Smale condition, $(PS)_c$ in short, for all $c\in (0, \beta)$,
\item [(ii)] there exist constants $\rho, \delta>0$ such that $J(u)\geq \delta$ for all $u\in H_2$ satisfying $\|u\|=\rho$.
\item [(iii)] there exists a $z\not \in H_1$ with $\|z\|=1$ and $R>\rho$ such that $J(u)\leq 0$ for all $u\in \partial Q$, where $Q=\{v+sz:
v\in H_1, \|v\|\leq R \;\;\text{and}\;\; 0\leq s\leq R\}$.
\end{itemize}
Then $c\in(0, \beta)$ defined as $c=\inf_{\eta\in \Gamma}\max_{u\in Q} J(\eta(u)) \geq \delta$, where $\Gamma=\{\eta \in C(\overline Q, H): \eta(u)=u,\;\;\text{if}\;\; u\in \partial Q\}$, is a critical value of $J$.
\end{thmlet}

In the following propositions we will show the above geometry.

\begin{pro}\label{LTBelow}
Let $\lambda_k<\lambda<\lambda_{k+1}$ and $f$ satisfies $(H_3)$. Then there exists $a, \rho>0$ such that
\[
\mathcal I_{A, \lambda}(u)\geq a,\;\;\text{for}\;\; \|u\|_{X_{0, A}}=\rho\;\;\text{and}\;\; u\in H_k^\perp
\]
\end{pro}
\begin{proof}
Proof follows the same lines as in Proposition \ref{MPBelow} using the characterization of $\lambda_{k+1}$. We remark that we do not require $(AR)$ condition in the proof.
\end{proof}
\begin{pro}\label{lg2}
Let $\lambda_k<\lambda<\lambda_{k+1}$ and $(H_1)$ holds. Define $Q=\{v+sz: v\in H_k, \|v\|\leq R \;\;\text{and}\;\; 0\leq s\leq R\;\;\text{for some }\;R>\rho\}$, where $\rho$ is given in Proposition~\ref{LTBelow} and $z \in W$ with $\|z\|_{X_{0,A}}=1$. Then
$\mathcal I_{A, \lambda}(u)\leq 0$ for all $u\in \partial Q$.
\end{pro}
\begin{proof}
	For some given $R>0$, let us split $\partial Q$ into following three parts
	\begin{align*}
	Q_1&=\{u\in H_k: \|u\|_{X_{0, A}}\leq R\};\\
	Q_2&=\{ u+sz: u\in H_k, \|u\|_{X_{0, A}}=R\;\;\text{and}\;0\leq s\leq R\};\\
	Q_3&=\{ u+Rz: u\in H_k, \|u\|_{X_{0, A}}\leq R\}.
	\end{align*}
	Now let us compute the energy functional in each of the above splitted boundary components of $Q$. If $u \in Q_1$, then using the characterization of $\lambda_k$ as in \eqref{EVL2}, we get
	\[
	\mathcal{I}_{A, \lambda}(u)\leq \frac{1}{2}\left(1-\frac{\lambda}{\lambda_k}\right)\|u\|_{X_{0, A}}^2\leq \frac{1}{2}\left(1-\frac{\lambda}{\lambda_k}\right)R^2<0
		\]
		for any choice of $R>0$. Before verifying the claim on $Q_2$ and $Q_3$, let us observe the following. Let us fix some $u_0\in H_k$. and introduce the scalar map $\psi: \mathbb R\to \mathbb R$ defined as $\psi(t)=\mathcal I_{A, \lambda}(tu_0)$. Now from a direct implication of assumption $(H_1)$, there exists $\mu>2$ and $C_1, C_2>0$ such that
		\begin{equation}\label{implicationH1Re}
		F(|u|)\geq C_1|u|^\mu-C_2 .
		\end{equation}
		Using \eqref{implicationH1Re} and using equivalence of $ X_{0, A}$ and $L^\mu$ norms, as $Q_2$ is finite dimensional $(\mathrm{dim}=k+1)$, we get
		\[
		\psi(t)\leq \frac{t^2}{2}\|u_0\|_{X_{0,A}}^2-\lambda \frac{t^2}{2}\|u_0\|_{X_{0,A}}^2-C_1 t^\mu \|u_0\|_{X_{0,A}}^\mu +2C_2
		\]
		which implies that $\psi(t)\to -\infty$ as $t\to \infty$. Now for any $u\in Q_2$, there exist $v\in H_k$ and $0\leq s\leq R$ such that
		  $u=v+s\omega$. Moreover,
		\[
		\|u\|_{X_{0, A}}^2=\|v+sz\|_{X_{0, A}}^2=\|v\|_{X_{0, A}}^2+s^2\|z\|_{X_{0, A}}^2\geq \|v\|^2_{X_{0, A}}=R^2.
		\]
		Therefore if we choose $R>0$ sufficiently large, we have $\mathcal I_{A, \lambda}(u)<0$. Now if $u \in Q_3$ then there exists some $v\in H_k$ such that $u=v+Rz$. Moreover,
		\[
		\|u\|_{X_{0, A}}^2=\|v+Rz\|_{X_{0, A}}^2=\|v\|_{X_{0, A}}^2+R^2\|z\|_{X_{0, A}}^2\geq R^2.
		\]
		Now following the similar argument as above one can prove the conclusion of the Proposition for choosing $R>0$ large enough.	
	\end{proof}

\subsection{The minimax level}

For the matter we have to select a $z \in W$ such that $\|z\|_{X_{0,A}}=1$ and $\mathcal{I}_{A, \lambda}(u)<\frac{\pi}{2\beta_0}$ for all $u \in Q$.


Let $P_k : X_{0,A} \rightarrow H_k^{\bot}$ be the orthogonal projection. Define
\begin{equation}\label{BEL32}
W_n(x) = P_k(M_n(x)).
\end{equation}
We need some estimates for $W_n$, which are shown in the next lemma. Before that,
knowing that $H_k$ has finite dimension, consider $A_k > 0$ and $B_k > 0$ such that
\begin{equation}\label{BEL33}
\|u\|_{X_{0,A}} \leq A_k \|u\|_2\,\, \textrm{and}\,\, \|u\|_\infty \leq \frac{B_k}{B} \|u\|_2,\,\, \textrm{for all}\,\, u \in H_k,
\end{equation}
where $B > 0$ is such that $\|M_n\|_2 \leq (\frac{B}{\log n})^{1/2}$
for all $n\in \mathbb N$.
\begin{lem}\label{lema511}
	Let $W_n$ be as defined in \eqref{BEL32}. Then the following estimates hold:
	\begin{description}
		\item[(i)] $1 - \frac{A_k}{\log n} \leq \|W_n\|^2_{X_{0,A}} \leq 1 + O((\log n)^{-1})$;
		\item[(ii)] $W_n(x) \geq \left\{
		\begin{array}{ll}
		\frac{-B_k}{\sqrt{\log n}}, & \hbox{for all}\,\, x \in (-1,1); \\
		\sqrt{\log n} - \frac{B_k}{\sqrt{\log n}}, & \hbox{for all}\,\, x \in (\frac{-1}{n},\frac{1}{n}).
		\end{array}
		\right.
		$
	\end{description}
\end{lem}

\begin{proof}
	To prove (i) one needs only to notice that
		\begin{align*}
	\|W_n\|^2_{X_{0,A}} &= \|M_n\|^2_{X_{0,A}}
	- \|(I - P_k)M_n\|^2_{X_{0,A}} \;\;\textrm{and}\;\;(I-P_k)M_n\in H_k.
	\end{align*}
	The estimate will follow because of \eqref{BEL33}.
	On the other hand, to verify $(ii)$, as $M_n \geq 0$ in $(-1,1)$ and $M_n = \sqrt{\log n}$ in
	$(\frac{-1}{n}, \frac{1}{n})$, we have
	\[
	W_n(x) = M_n(x) -(I-P_k)(M_n(x)) \geq \left\{
	\begin{array}{ll}
	-\|(I-P_k)M_n\|_\infty, & \hbox{if}\,\, x \in (-1,1); \\
	\sqrt{\log n} - \|(I-P_k)M_n\|_\infty, & \hbox{if}\,\, (\frac{-1}{n}, \frac{1}{n}),
	\end{array}
	\right.
	\]
	where the inequality follows by observing the definition of $B_k$ in \eqref{BEL33}.
\end{proof}

\medskip

In the following, we define $z_n(x)=\frac{W_n(x)}{\|W_n\|_{X_{0, A}}}$, $Q_n=\{v+sz_n: v\in H_k, \|v\|\leq R \;\;\text{and}\;\; 0\leq s\leq R\;\;\text{for some }\;R>\rho\}$ and the minimax level of $\mathcal I_{A, \lambda}$ as follows
\begin{equation}\label{LinLevel}
c(n)=\inf_{\nu\in \Gamma} \sup_{w\in \nu(Q_n)}\mathcal I_{A, \lambda}(w),
\end{equation}
where
 \[
\Gamma=\{ \nu\in C(Q_n, H):\nu(w)=w \;\text{if} \; w\in \partial Q_n\}.
\]
We have the following estimate for above minimax level.
\begin{pro}\label{minimaxlev2}
Let $c(n)$ be given as in \eqref{LinLevel} and assumption $(H_1)-(H_5)$ hold. Then for large $n$, $c(n)<\frac{\pi}{2\beta_0}$.
\end{pro}
\begin{proof}
	From the definition of $c(n)$, it is enough to show that
	\begin{equation}\label{de54}
	max\{\mathcal I_{A, \lambda}(v+sz_{n}): v\in H_k, \|v\|_{X_{0, A}}\leq R, 0\leq s\leq R\}<\frac{\pi}{2\beta_0}.
	\end{equation}
	Let us proceed by contradiction. Suppose that \eqref{de54} does not hold, then
	\[
	max\{\mathcal I_{A, \lambda}(v+sz_{n}): v\in H_k, \|v\|_{X_{0, A}}\leq R, 0\leq s\leq R\}\geq\frac{\pi}{2\beta_0}.
	\]
	Let $u_n=v_n+s_nz_n$ be the point of maximum in the above expression with $v_n \in H_k$. Then
	
	\begin{equation}\label{contrassump}
 \mathcal I_{A, \lambda}(v_n+s_nz_{n})\geq \frac{\pi}{2\beta_0}.
	\end{equation}
	Moreover, since $\mathcal I_{A, \lambda}^\prime (u_n)=0$,  we have
	
\begin{equation}\label{final1}
	\|u_n\|_{X_{0, A}}^2-\lambda \int_\Omega |u_n|^2 dx-\int_\Omega f(|u_n|)|u_n|^2dx=0.
\end{equation}
	Now we finish the proof of the Proposition in following few steps.\\
	
\noindent\textbf{Step 1:}  We claim that $\{v_n\}$ and $\{s_n\}$ are bounded sequences in respective topologies.
	\begin{proof}
		There are either of the following two posibilities
		\begin{itemize}
			\item [(i)] $\dfrac{s_n}{\|v_n\|_{X_{0, A}}}\geq C_0$ for some $C_0>0$ uniformly in $n$.
			\item[(ii)] $\dfrac{s_n}{\|v_n\|_{X_{0, A}}}\to 0$ in $\mathbb R$, up to a subsequence, as $n\to \infty$.
			\end{itemize}

		\medskip
		
		Suppose $(i)$ holds true. Note that the boundedness of $\{s_n\}$ implies the sequence $\{v_n\}$ is also bounded as $\|v_n\|_{X_{0, A}}\leq s_n/C_0$. Hence, we aim to prove the boundedness of $s_n$ in light of item $(i)$ as above. For, there exists a constant $C$ such that
		\[
\|u_n\|_{X_{0, A}}=\|v_n+s_nz_n\|_{X_{0, A}}\leq \|v_n\|_{X_{0, A}}+s_n\|z_n\|_{X_{0, A}}\leq \frac{s_n}{C_0}+s_n\leq Cs_n.						
		\]
		Now from \eqref{final1} and assumption $(H_4)$, given $\zeta>0$, there exists $t_\zeta>0$ large enough such that $f(t)t^2\geq \zeta e^{\beta_0t^2}$ for all $t>t_\zeta$, we get
		\begin{equation}\label{bdd2}
	Cs_n^2\geq	\int_{B_{1/n}\cap \{|u_n|\geq t_\zeta\}}f(|u_n|)|u_n|^2 \ud x\geq \zeta\int_{B_{1/n}\cap \{|u_n|\geq t_\zeta\}} e^{\beta_0 |u_n|^2} \ud x.
		\end{equation}
		Now to estimate the integral in the above inequality in right hand side, from Lemma \ref{lema511}, we have in $B_{1/n}$ for large $n$ and $\epsilon \in (0,1)$, that
	\[
	u_n(x)=v_n(x)+s_nz_n(x)=\frac{s_n(\sqrt\log n-\frac{B_k}{\log n})}{\|M_n\|_{X_{0, A}}}\left(\frac{v_n(x)\|M_n\|_{X_{0, A}}}{s_n (\sqrt\log n-\frac{B_k}{\log n})}+1\right)\geq (1-\epsilon)\frac{s_n\sqrt\log n}{\|M_n\|_{X_{0, A}}}.
	\]
	Hence from \eqref{bdd2} and using $\|M_n\|_{X_{0,A}}^2 \leq \pi + C(\log n)^{-1}$, we get
	\[
	Cs_n^2\geq \zeta\int_{B_{1/n}} e^{\frac{\beta_0(1-\epsilon)^2 s_n^2\log n}{\|M_n\|^2_{X_{0, A}}} } \ud x= 2\zeta e^{\left(\frac{\beta_0(1-\epsilon)^2 s_n^2}{\pi + C(\log n)^{-1}}-1\right)\log n},
	\]
	which implies
	\begin{equation}\label{conclusive}
		C s_n^2\geq 2 \zeta e^{\left(\frac{\beta_0(1-\epsilon)^2 s_n^2}{\pi+ C(\log n)^{-1}}-1\right)\log n}.
\end{equation}
Therefore if $s_n\to \infty$, it contradicts the above inequality. Hence $\{s_n\}$ is a bounded sequence so is $\{v_n\}$.

\medskip
	
	Next we assume that $(ii)$ occurs. Then $s_n\leq \|v_n\|_{X_{0, A}}$ which implies $\|u_n\|_{X_{0, A}}=\|v_n+s_nz_n\|_{X_{0, A}}\leq 2\|v_n\|_{X_{0, A}}$. Note that if the sequence $\{\|v_n\|\}$ is bounded in $X_{0, A}$ then the sequence $\{s_n\}$ is bounded in $\mathbb R$. Thus we aim to show that $\{\|v_n\|\}$ is bounded in $X_{0, A}$. Assume by contradiction that $\|v_n\|_{X_{0, A}} \to \infty$. From \eqref{final1}, we have
	\begin{equation}\label{claim1sec}
	1\geq \int_{\{|u_n|>t_\zeta\}}\frac{f(|u_n|)|u_n|^2}{\|u_n\|_{X_{0, A}}^2}\ud x\geq \frac{\zeta}{4}\int_{\{|u_n|>t_\zeta\}}\frac{e^{\beta_0 |u_n|^2}}{\|v_n\|_{X_{0, A}}^2} \ud x.
	\end{equation}
	Observe that
	\begin{equation*}\label{pointwise}
	\frac{u_n}{\|v_n\|_{X_{0, A}}}\chi_{\{u_n\geq t_\zeta\}}=\frac{v_n}{\|v_n\|_{X_{0, A}}}+\frac{s_n}{\|v_n\|_{X_{0, A}}}z_n-\frac{u_n}{\|v_n\|_{X_{0, A}}}\chi_{\{u_n<t_\zeta\}}.
	\end{equation*}
	Since, $\dfrac{s_n}{\|v_n\|_{X_{0, A}}}\to 0$ in $\mathbb R$ and $z_n\to 0$ pointwise almost everywhere in $(-1, 1)$, there exists $v_0\in H_k$ such that
	\begin{equation}\label{pointwise2}
		\frac{u_n(x)}{\|v_n\|_{X_{0, A}}}\chi_{\{u_n\geq t_\zeta\}}\to v_0\;\; a. e. \; \text{in}\;\; (-1, 1)
	\end{equation}
	with $\dfrac{v_n}{\|v_n\|_{X_{0, A}}}\to v_0$ and $\|v_0\|_{X_{0, A}}=1$. Then using $\|v_n\|_{X_{0, A}} \to +\infty$, \eqref{claim1sec}, \eqref{pointwise2} and Fatou's Lemma, we get
	\[
	1\geq\frac{\zeta}{4}\int_\Omega\frac{e^{\beta_0 \|v_n\|_{X_{0, A}}^2 (\frac{u_n}{\|v_n\|_{X_{0, A}}}\chi_{\{u_n\geq t_\epsilon\}})^2}}{\|v_n\|_{X_{0, A}}^2} \ud x \to +\infty\,\,\mbox{as}\,\, n \to +\infty,
	\]
	which is a contradiction. Hence the proof of the claim.
		\end{proof}
	
	\medskip
	
\noindent\textbf{Step 2:} From step 1, we can assume that there exists $v_0\in X_{0, A}$ and $s_0\in \mathbb R$ such that $v_n\to v_0$ and $s_n\to s_0$, up to a subsequence. Now we claim that $v_0=0$ and $s_0^2=\frac{\pi}{\beta_0}$.
	\begin{proof}
		First we show that $s_0^2\geq \frac{\pi}{\beta_0}$. By the definition
		\[
		\|u_n\|_{X_{0,A}}^2\to \|v_0\|_{X_{0,A}}^2+s_0^2.
		\]
		Moreover, using $v_n\to v_0$ in $X_{0, A}$, $\|z_n\|_2\to 0$ and Cauchy Schwartz inequality together with embeddings of $X_{0, A}\hookrightarrow L^2((-1, 1))$, we get $u_n\to v_0$ in $L^1((-1, 1)).$ From \eqref{final1}, we obtain
		\[
		\int_\Omega f(|u_n|)|u_n|^2\, \ud x \leq C.
		\]
	Consequently,	from \cite[Lemma 2.1]{FMR}, we get $\int_{\Omega} f(|u_n|)|u_n|\ud x \rightarrow \int_{\Omega} f(|v_0|)|v_0|\ud x$.
		Thus, by applying $(H_1)$ and the Generalized Lebesgue Dominated Convergence Theorem we have
		\begin{equation}\label{djconv}
		\int_\Omega F(|u_n|)\ud x\to \int_\Omega F(|v_0|) \ud x.
		\end{equation}
		In light of \eqref{djconv}, \eqref{contrassump} and \eqref{EVL2}, we get
		\[
		 \frac{\pi}{2\beta_0}\leq \lim_{n\to \infty}\mathcal I_{A, \lambda}(u_n)= \frac{1}{2}\|v_0\|_{X_{0, A}}^2-\frac{\lambda}{2}\|v_0\|_2^2-\int_\Omega F(|v_0|)\ud x+\frac{s_0^2}{2}\leq \frac{1}{2}\left(1-\frac{\lambda}{\lambda_k}\right)\|v_0\|^2_{X_{0, A}}+\frac{s_0^2}{2}.
		\]
		Since $\lambda \in (\lambda_k, \lambda_{k+1})$, $s_0^2\geq {\pi}/{\beta_0}$.
		
		\medskip
		
		Now we follow the idea of alternatives as in the step 1. Note that, since $s_0^2\geq {\pi}/{\beta_0}$ and $\|v_n\|_{X_{0, A}}\leq C$ for some $C>0$, the alternative $(ii)$ is not possible to hold. Hence suppose $(i)$ holds. Then, from \eqref{conclusive}, we have
		\[
		\frac{\beta_0(1-\epsilon)^2 s_0^2}{\pi}-1\leq 0
		\]
		which implies $s_0^2\leq \pi/\beta_0$. Hence the proof.
		
		\medskip
		
		Next we show that $v_0\equiv 0$.
		From \eqref{contrassump}, using \eqref{djconv},  $v_n\to v_0$ in $X_{0, A}$, $\|z_n\|_{X_{0, A}}=1$, $\|z_n\|_2\to 0$ and $s_n\to s_0$ in $\mathbb R$, we get the following
	\[
	\lim_{n\to \infty}\mathcal I_{A, \lambda}(u_n)=\mathcal I_{A, \lambda}(v_0)+\frac{s_0^2}{2}\geq \frac{\pi}{2\beta_0}
	\]
	which implies that
	\[
	 \mathcal I_{A, \lambda}(v_0)\geq 0.
	\]
	Moreover, by the definition,
	\[
	\mathcal I_{A, \lambda}(v_0)=\frac{1}{2}\|v_0\|_{X_{0, A}}^2-\frac{\lambda}{2}\|v_0\|_2^2-\int_\Omega F(|v_0|) \ud x\leq \frac{1}{2}\|v_0\|_{X_{0, A}}^2-\frac{\lambda}{2}\|v_0\|_2^2\leq \frac{1}{2}\left(1-\frac{\lambda}{\lambda_k}\right)\|v_0\|^2_{X_{0, A}}\leq 0.
	\]
	Hence from above two inequalities $\mathcal I_{A, \lambda}(v_0)=0$. Since $v_0\in H_k$, we have
	\[
	0=\mathcal I_{A, \lambda}(v_0)=\frac{1}{2}\left(1-\frac{\lambda}{\lambda_k}\right)\|v_0\|^2_{X_{0, A}}-\int_\Omega F(|v_0|) \ud x\leq -\int_\Omega F(|v_0|) \ud x
	\]
	and by the nature of the nonlinearity
	\[
	\int_\Omega F(|v_0|) \ud x\geq 0.
	\]
	On combining these two estimates, we have
	\[
	\int_\Omega F(|v_0|) \ud x=0
	\]
	which implies, from $\mathcal I_{A, \lambda}(v_0)=0$,  that $\|v_0\|_{X_{0, A}}=0$. It completes the proof.
		\end{proof}	
	
	\medskip

In order to conclude the proof, observe that from step 2, up to a subsequence, we have $v_n\to 0$ strongly in $H_k$ and $s_n\to s_0$. Then \eqref{conclusive} holds, that is, 
\[
C s_n^2\geq 2 \zeta e^{\left(\frac{\beta_0(1-\epsilon)^2 s_n^2}{\pi+ C(\log n)^{-1}}-1\right)\log n}.
\]
Letting $\epsilon \to 0^+$ and later $n \to +\infty$, we get
$$\zeta \leq \dfrac{C \pi e^{\frac{C}{\pi}}}{2 \beta_0}.$$
Since $\zeta$ is arbitrarily large, we get a contradiction.
This completes the proof of the proposition.
		\end{proof}

\subsection{Proof of Theorem~\ref{thm2}}

To conclude Theorem \ref{thm2} we use Propositions \ref{LTBelow}, \ref{lg2} and \ref{minimaxlev2},
and apply the Theorem B.

\section{Proof of Theorem \ref{thm3}}

In order to prove Theorem \ref{thm3}, we will use the following critical point theorem, see \cite[Theorem~2.4]{crlit}.

\begin{thmlet}\label{GCPT}
Let $H$ be a real Hilbert space with the induced norm $\|\cdot\|$ and $J:H\rightarrow \mathbb R$  be a functional of class $C^1(H, \mathbb R)$
satisfying the following conditions:

\begin{description}
  \item[$(A_1)$] $J(0)=0$ and $J(-u)=J(u)$;
  \item[$(A_2)$] $J$ satisfies the Palais-Smale condition, in short $(PS)_c$, for $c\in (0, \beta)$ and for some $\beta>0$;
  \item[$(A_3)$] there exist closed subspaces $V, W$ of $H$ and constants $\rho, \delta, \eta$ with $\delta<\eta<\beta$ such that
  \begin{description}
    \item[$(i)$] $J(u)\leq \eta$ for all $u\in W$;
    \item[$(ii)$] $J(u)\geq \delta$ for any $u\in V$ with $\|u\|=\rho$;
    \item[$(iii)$] $codim~(V)<\infty$ and $dim~W\geq codim~V$.
  \end{description}
\end{description}
Then there exist at least $dim~W-codim~V$ pairs of critical points of the functional $J$ with critical values belonging to the interval $[\delta, \eta]$.
\end{thmlet}

Our next aim is to apply Theorem \ref{GCPT} in our variational setup. It is clear that the functional $\mathcal I_{A,\lambda} \in C^1(X_{0, A}, \mathbb R)$ and from the definition, $\mathcal I_{A,\lambda}(0)=0$. Since $|-u|=|u|$ implies $\mathcal I_{A,\lambda}(-u)=\mathcal I_{A,\lambda}(u)$. Hence the assumption $(A_1)$ is satisfied. Lemma \ref{PS} implies that $I_{A, \lambda}$ satisfies the $(PS)_c$ condition for all $c\in (0, \frac{\pi}{2\beta_0} )$. Hence the assumption $(A_2)$ holds good with $\beta=\frac{\pi}{2\beta_0}$. Next we verify the assumption $(A_3)$.
We consider $W= \displaystyle{\rm{span}_{\mathbb R} \{\varphi_1, \varphi _2,.....\varphi_{k+m-1}\}}$ and
 \[
 V= X_{0, A} \;\;\text{if}\;\; k=1, \;\;\text{otherwise}\;\;
  V=\{u\in X_{0, A}: \langle u,\varphi_j\rangle=0\;\;\forall\; 1\leq j\leq k-1\}.
  \]
Then both $W$ and $V$ are closed subspaces of $X_{0, A}$ with $k+m-1=$ dim $W\geq$ codim $V=k-1$.
Now take $u\in W$ then $u(x)=\displaystyle \sum_{j=1}^{k+m-1}\alpha_j\varphi_j(x)$ and by the orthogonality of eigenfunctions
\begin{equation*}
\|u\|^2_{X_{0,A}}=\displaystyle \sum_{j=1}^{k+m-1}\alpha_j^2\|\varphi_j\|^2_{X_{0,A}}=\displaystyle \sum_{j=1}^{k+m-1}\lambda_j\alpha_j^2\leq \lambda_k\displaystyle \sum_{j=1}^{k+m-1}\alpha_j^2=\lambda_k\|u\|_2^2=\lambda^*\|u\|_2^2.
\end{equation*}
Now using $(H_5)$, we have $F(|t|)\geq\frac{C_p}{p}|t|^p$, for all $t \in \mathbb{R}$. Thus for $u\in W$
\begin{align*}
\mathcal I_{A,\lambda}(u)&=\frac{1}{2}\|u\|^2_{X_{0,A}}-\frac{\lambda}{2}\|u\|_2^2-\int_\Omega F(|u|)\, \ud x\\ &\leq \frac{1}{2}(\lambda^*-\lambda)\|u\|_2^2-\int_\Omega F(|u|)\, \ud x\\
&\leq \frac{1}{2}(\lambda^*-\lambda)2^{\frac{p}{p-2}}\|u\|_p^2-\frac{C_p}{p}\|u\|_p^p.
\end{align*}
Define $h(t)=\frac{1}{2}(\lambda^*-\lambda)2^{\frac{p}{p-2}}t^2-\frac{C_p}{p}t^p$ for $t\geq 0$, then $h(t)$ has a maximum at $t_0=\left(\frac{(\lambda^*-\lambda)}{C_p}2^{\frac{p}{p-2}}\right)^\frac{1}{p-2}$. Hence
\[
I_{A,\lambda}(u)\leq \eta=\left(\frac{1}{2}-\frac{1}{p}\right)\left(\frac{(\lambda^*-\lambda)}{C_p^{2/p}}2^{\frac{p}{p-2}}\right)^\frac{p}{p-2}.\]
Note that we can make $\eta$ to be arbitrary small positive number either by choosing $\lambda$ suitably close to $\lambda^*$ or by taking $C_p>0$ large enough in $(H_5)$. We will determine this closeness later.\\
For the second part, we use Proposition \ref{MPBelow} and Proposition \ref{LTBelow}.
Now only thing remains to show is the relation
\begin{equation*}
\label{relation}
\delta<\eta<\beta.
\end{equation*}
Note that the first inequality can be justified by choosing $\rho$ sufficiently small to make $\delta>0$ arbitrary small in Proposition \ref{MPBelow} or in Proposition \ref{LTBelow}. Hence $\delta<\eta$ holds good for $\rho>0$ sufficiently small. The ultimate task is to show that $\eta<\beta$. In other words,
\[
0<\left(\frac{1}{2}-\frac{1}{p}\right)\left(\frac{(\lambda^*-\lambda)}{C_p^{2/p}}2^{\frac{p}{p-2}}\right)^\frac{p}{p-2}<\frac{\pi}{2\beta_0}
\]
which leads to a restriction on  $\lambda$ and $C_p$ as $\lambda<\lambda^*$ and
\[
C_p>\left( \frac{\beta_0(p-2)}{\pi}\right)^\frac{p-2}{2}\left((\lambda^*-\lambda)2^{\frac{p}{p-2}}\right)^\frac{p}{2}>0
\]
and therefore justifies the choices of $\lambda$ and $C_p$ as in Theorem \ref{thm3}. Hence the proof of Theorem \ref{thm3}.

\section{Proof of Theorem \ref{thm4}}

The proof of Theorem \ref{thm4} is mainly based on the application of the following result due to \cite[Theorem~6]{Ricceri}.

\begin{thmlet}\label{Recci}
Let $(H, \|\cdot\|)$ be a real reflexive Banach space  and $\Phi, \Psi: H\rightarrow \mathbb R$ be two continuously Gateaux differentiable functionals such that $\Phi$ is sequentially weakly lower semicontinuuous and coercive. Further assume that $\Psi$ is sequentially weakly continuous. In addition, assume that, for each $\gamma>0$ the functional $\mathcal I_\gamma: H\to \mathbb R$
\[
\mathcal I_\gamma(z):=\gamma \Phi(z)-\Psi(z), \,\, z\in H,
\]
 satisfies $(PS)_c$ condition for all $c\in \mathbb R$. Then for any $\rho>\inf_H \Psi$ and every
 \[
 \gamma>\inf_{u\in \Phi^{-1}(-\infty, \;\rho)}\frac{\sup_{v\in \Phi^{-1}(-\infty, \;\rho)}\Psi(v)-\Psi(u)}{\rho-\Phi(u)}
 \]
 the following alternative holds: either the functional $\mathcal I_\gamma$ has a strict global minimum in $\Phi^{-1}(-\infty, \;\rho)$, or $\mathcal I_\gamma$ has at least two critical points one of which lies in $\Phi^{-1}(-\infty, \;\rho)$.
\end{thmlet}
Here we consider the functional $E_{A, \lambda}: X_{0, A}\to \mathbb R$ as
\begin{equation}\label{Energy}
E_{A, \lambda}(u)=\frac{1}{\lambda}J(u)-K_\lambda(u),
\end{equation}
where
\begin{equation*}
J(u)=\frac12\|u\|^2_{X_{0, A}}\;\;\text{and}\;\; K_\lambda(u)=\frac12 \|u\|_2^2+\frac{1}{\lambda}\int_\Omega F(|u|)\, \ud x.
\end{equation*}
It is straightforward to see that $J$ is continuously Gateaux differentiable, sequentially weakly lower semicontinuuous and coercive.

\begin{lem}
If $f$ satisfies $(H_1)-(H_2)$ and \eqref{subgrowth}, then $K_\lambda$ is continuously Gateaux differentiable, sequentially weakly lower semicontinuuous and coercive.
\end{lem}
\begin{proof}
Since $f$ has subcritical growth the proof is easy and we will omit it.
Moreover, $K_\lambda$ is sequentially weakly continuous.
\end{proof}
The next result is about the Palais-Smale condition.
\begin{pro}
If $f$ satisfies $(H_1)$ and \eqref{subgrowth}, then the functional defined in \eqref{Energy} satisfies $(PS)_c$  for all $c\in \mathbb R$.
\end{pro}
\begin{proof}
Let us consider $\{u_n\}$ be a Palais-Smale sequence for the functional $E_{\lambda, A}$, that is, \begin{equation}\label{EJ23R}
\frac{1}{2\lambda}\|u_n\|^2_{X_{0,A}} - \frac{1}{2}\|u_n\|_2^2 - \frac{1}{\lambda}\int_{\Omega} F(|u_n|) \ud x \to c,\,\, \mbox{as}\,\, n \to +\infty,
\end{equation}
and
\begin{equation}\label{EJ24R}
\left|\frac{1}{\lambda}\Re\langle u_n, v \rangle_{X_{0,A}} -  \Re\langle u_n , v\rangle_{L^2} - \frac{1}{\lambda}\Re\int_{\Omega} f(|u_n|) u_n \overline{v} \,\ud x \right| \leq \varepsilon_n \|v\|_{X_{0,A}},\,\, \mbox{for all}\,\,
v \in X_{0,A}.
\end{equation} To prove the claim of the above proposition, we divide the proof into a few steps.\\

\noindent\textbf{Step 1:} The Palais-Smale sequence is bounded.\\

\noindent The proof of this step follows the same lines as in Lemma \ref{PS}. Consequently, there exists $u_o\in X_{0, A}$ such that $u_n\rightharpoonup u_o$ weakly in $X_{0, A}$, $u_k\to u_o$ in $L^q(\Omega)$ for all $q\in [1, \infty)$ and $u_n(x)\to u_o(x)$ a.e. in $(\Omega)$.
\medskip

\noindent \textbf{Step 2:} The following convergence holds
\[
\int_\Omega f(|u_n|)|u_n|^2 \ud x \to \int_\Omega f(|u_o|)|u_o|^2 \ud x.
\]
To prove the claim of this step, we proceed as follows
\begin{align*}
\left|\int_\Omega f(|u_n|)|u_n|^2 \ud x- \int_\Omega f(|u_o|)|u_o|^2 \ud x\right|&\leq \left|\int_\Omega f(|u_n|)|u_n|(|u_n|-|u_o|)\, \ud x- \int_\Omega (f(|u_n|)|u_n|- f(|u_o|)|u_o|)|u_o| \,\ud x\right|.
\end{align*}
Let us denote
\[
I_1=\int_\Omega |f(|u_n|)|~|u_n|~\Big||u_n|-|u_o|\big| \,\ud x\;\;\text{and}\;\; I_2=\int_\Omega \Big(f(|u_n|)|u_n|- f(|u_o|)|u_o|\Big)|u_o| \,\ud x.
\]
We estimate these integrals one by one as follows. We begin with $I_1$, by using the estimate \eqref{subgrowth}, the elementary inequality $\Big||a|-|b|\Big|\leq |a-b|$ for $a, b\in \mathbb C$ and H\"older's inequality as
\[
I_1\leq C\int_\Omega e^{\beta |u_n|^2}|u_n-u_o| \ud x\leq C\left(\int_\Omega e^{q\beta \|u_n\|^2_{X_{0, A}}\left(\frac{|u_n|}{\|u_n\|_{X_{0, A}}}\right)^2} \ud x
\right)^\frac{1}{q}\left(\int_\Omega|u_n-u_o|^p\ud x\right)^\frac{1}{p}.
\]
Now using the fact that $u_n\to u_o$ in $L^q(\Omega)$ for all $q\in [1, \infty$) and $q~\beta~\|u_n\|_{X_{0, A}}^2<\pi$ for suitable choosen $\beta>0$ we get that $I_1\to 0$ as $n\to \infty$.
Next we show the similar convergence for $I_2$.  The proof of this convergence follows from the Lemma 2.1 of \cite{FMR} once $\int_\Omega f(|u_n|)|u_n|^2\ud x<C_1$ which follows from \eqref{subgrowth} and  boundedness of the sequence $\{u_n\}$. Hence $I_2\to 0$ as $n\to \infty$. Consequently the claim of the Step 2 is proved.
\\

\noindent\textbf{Step 3:} Up to a subsequence, $u_n\to u_o$ in $X_{0, A}$.\\

\noindent Take $v={u}_o$ in \eqref{EJ23R},  we get
\begin{equation}\label{rufl1}
\frac{1}{\lambda}\|u_0\|^2_{X_{0,A}} -  \|u_o\|_2^2 - \frac{1}{\lambda}\int_\Omega f(|u_o|) |u_o|^2  \ud x  =0.
\end{equation}
On the other hand, if we take $v={u}_n$ in \eqref{EJ24R} and use Step 2, we get
\begin{equation}\label{rufl2}
\frac{1}{\lambda}\|u_n\|^2_{X_{0,A}} - \|u_o\|_2^2 - \frac{1}{\lambda}\int_\Omega f(|u_o|) |u_o|^2  \ud x  \to 0.
\end{equation}
From \eqref{rufl1} and \eqref{rufl2}, we have $\|u_n\|^2 \to \|u_o\|^2$ in $\mathbb R$ and hence  $u_n\to u_o$ in $X_{0, A}$.
\end{proof}

\subsection{Justification for the choice of $\lambda$}

In the statement of Theorem \ref{Recci}, it can be noticed that the result holds good for all $\rho>0$ in light of the definition of $J$.
Now we define the range of $\lambda$ as follows.
\[
\frac{1}{\lambda}>\Theta_\lambda,
\,\,\text{where}\,\,\Theta_\lambda=\inf_{u\in J^{-1}(-\infty, \;\rho)}\frac{\sup_{v\in J^{-1}(-\infty, \;\rho)}K_\lambda(v)-K_\lambda(u)}{\rho-J(u)}.
\]
Since $J(0)=0=K_\lambda(0)$,
\[
\Theta_\lambda\leq \frac{1}{\rho}\sup_{v\in J^{-1}(-\infty, \;\rho)}K_\lambda(v)=\frac{1}{2\rho}\sup_{\{v\in X_{0, A}\,:\, \|v\|_{X_{0,A}}\leq (2\rho)^\frac12\}}\|v\|_2^2+\frac{1}{\rho\lambda}\sup_{\{v\in X_{0, A}\,:\, \|v\|_{X_{0,A}}\leq (2\rho)^\frac12\}}\int_\Omega F(|v|)\, \ud x.
\]
On the other hand, using \eqref{Holder}
\begin{equation}\label{gned1}
\sup_{\{v\in X_{0, A}\,:\, \|v\|_{X_{0,A}}\leq (2\rho)^\frac12\}}\|v\|_2^2\leq 2\rho S_2^2.
\end{equation}
Under the assumption $(H_2)$   and \eqref{subgrowth}, we can get the following estimate
 \[
 \sup_{\{v\in X_{0, A}\,:\, \|v\|_{X_{0,A}}\leq (2\rho)^\frac12\}}\int_\Omega F(|v|)\, \ud x\leq  \sup_{\{v\in X_{0, A}\,:\, \|v\|_{X_{0,A}}\leq (2\rho)^\frac12\}}\int_\Omega f(|v|)|v|^2 \, \ud x \leq C \sup_{\{v\in X_{0, A}\,:\, \|v\|_{X_{0,A}}\leq (2\rho)^\frac12\}}\int_\Omega e^{\beta |v|^2}|v|\, \ud x.
 \]
 Now using H\"{o}lder's inequality with conjugate exponents $1/p+1/q=1$ in the last term, we get
 \[
 \sup_{\{v\in X_{0, A}\,:\, \|v\|_{X_{0,A}}\leq (2\rho)^\frac12\}}\int_\Omega e^{\beta |v|^2}|v|\, \ud x\leq
 \sup_{\{v\in X_{0, A}\,:\, \|v\|_{X_{0,A}}\leq (2\rho)^\frac12\}}\left(\int_\Omega e^{q\beta |v|^2}\, \ud x\right)^\frac{1}{q}
 \left(\int_\Omega |v|^p \ud x\right)^\frac{1}{p}.
 \]
 Again recalling \eqref{Holder} and choosing $\beta>0$ sufficiently small such that $\beta q\|v\|_{X_{0, A}}^2<\pi$, by Trudinger-Moser inequality \eqref{Corolario1}, we get
\begin{equation}\label{gned2}
 \sup_{\{v\in X_{0, A}\,:\, \|v\|_{X_{0,A}}\leq (2\rho)^\frac12\}}\int_\Omega F(|v|) \ud x\leq C_1(\beta)\|v\|_p\leq C_\beta (2\rho)^\frac12 S_p\,.
 \end{equation}
 By combining \eqref{gned1} and \eqref{gned2}, we have the following estimate for $\Theta_\lambda$
 \[
 \Theta_\lambda\leq S_2^2+\frac{\sqrt{2}C_\beta S_p}{\lambda\sqrt{\rho}}.
 \]
 Therefore if we choose
 \[
0< \lambda<\Lambda(\rho):=\frac{1}{S_2^2}\left(1-\frac{\sqrt{2}C_\beta S_p}{\sqrt{\rho}}\right)
 \]
 for any $\rho>2C_\beta^2S_p^2$, we can justify the choice of $\lambda$ as in Theorem \ref{thm4}.
 Now only thing remain to show that the possibility of global minima  for the functional $E_{\lambda,\;A}$ will not occur.
From assumption $(H_1)$, there exists $\mu>2$ and $C_1, C_2>0$ such that
\begin{equation}\label{implicationHF}
F(|u|)\geq C_1|u|^\mu-C_2 .
\end{equation}
Using \eqref{implicationHF}, we estimate for arbitrary but fixed $u\in X_{0, A}$
\[
E_{A, \lambda}(t u)\leq \frac{t^2}{2\lambda}\|u\|_{X_{0, A}}^2-\frac{t^2}{2}\|u\|_2^2-\frac{C_1 t^\mu}{\lambda}\|u\|_\mu^\mu+\frac{2C_2}{\lambda}
\]
which implies that $E_{A, \lambda}(tu)\to -\infty$ as $t\to \infty$.
Hence $E_{A, \lambda}$ cannot have a strict global minimum in $J^{-1}(-\infty, \;\rho)$.

\section*{Acknowledgement} 
\noindent Research supported in part by INCTmat/MCT/Brazil, CNPq and CAPES/Brazil.


\section*{References}
\begin{thebibliography}{99}
	
\bibitem{FL1}	W. Abdelhedi and H. Chtioui, \emph{On a Nirenberg-type problem involving the square root of the Laplacian}, Journal of Functional Analysis, 265 (2013), 2937--2955.

		\bibitem{Adimurthi} Adimurthi, \emph{Existence of positive solutions of the semilinear Dirichlet problem with critical growth for the n-Laplacian}, Ann.
	Sc. Norm. Super. Pisa Cl. Sci. 17 (1990), 393–413.
	
\bibitem{AR1973}
A. Ambrosetti and P. H. Rabinowitz, \emph{Dual variational methods in critical point theory and applications}, J. Funct. Anal. 14 (1973), 349–381.
	
\bibitem{FL2} V. Ambrosio, G. M. Bisci, and D. Repov\v{s}, \emph{Nonlinear equations involving the square root of the Laplacian},  Discrete \& Continuous Dynamical Systems - S, 12 (2019), 151-170.
 


\bibitem{AA} 
V. Ambrosio and P. d'Avenia, \emph{Nonlinear fractional magnetic Schr\"odinger equation: existence and multiplicity}, J.
Differential Equations 264 (2018), 3336–3368.

\bibitem{Ambrossio} V. Ambrosio, \emph{On a fractional magnetic Schrödinger equation in $\mathbb R$ withexponential critical growth}, Nonlinear Analysis 183 (2019), 117-148.

\bibitem{Applbm} D. Applebaum, \emph{L\'evy processes-from probability to finance and quantum groups}, Notices Am. Math. Soc. 51 (2004), 1336–1347.

\bibitem{crlit} 
P. Bartolo and V. Benci, \emph{Abstract critical point theorems and applications to some nonlinear problems with strong resonance at infinity}, Nonlinear analysis: Theory, methods \& applications 7 (1983), 981-1012.


 \bibitem{FL3} X. Cabr\'e and J. Tan, \emph{Positive solutions of nonlinear problems involving the square root of the Laplacian}, Advances in Mathematics 224 (2010), 2052--2093.
  \bibitem{Cao}
D. M. Cao,
\emph{nontrivial solution of semilinear elliptic
	equation with critical exponent in $\mathbb{R}^2$},
Comm. Partial Diff. Eq. 17 (1992), 407-435.


\bibitem{sfs} 
G. Cerami, D. Fortuno and M. Struwe, \emph{Bifurcation and multiplicity results for nonlinear elliptic problems involving critical Sobolev exponents}, Ann. Inst. Henri poincare 1 (1984), 341-350.

\bibitem{App1} W. Chen and S. Holm, \emph{Fractional Laplacian time-space models for linear and nonlinear lossy media exhibiting arbitrary frequency power-law dependency}, The Journal of the Acoustical Society of America 115 (2004), 1424-–1430.
\bibitem{App2} R. Cont and E. Voltchkova, \emph{A finite difference scheme for option pricing in jump diffusion and exponential lévy models},  SIAM Journal on Numerical Analysis 43 (2005), 1596–-1626.


\bibitem{AS} 
P. d'Avenia and M. Squassina, \emph{Ground states for fractional magnetic operators}, ESAIM Control Optim. Calc. Var. 24 (2018), 1--24.


\bibitem{FMR}
D. G. de Figueiredo, O. H. Miyagaki and B. Ruf, \textit{Elliptic
equations in $\mathbb{R}^2$ with nonlinearities in the critical
growth range}, Calc. Var. Partial Differential Equations
3 (1995), 139--153.

\bibitem{MR3399183}
J.~M. do~\'{O}, O.~H. Miyagaki, and M.~Squassina.
\emph{Nonautonomous fractional problems with exponential growth},
NoDEA Nonlinear Differential Equations Appl. 22 (2015), 1395--1410.


\bibitem{App3} B. P. Epps and B. Cushman-Roisin,  \emph{Turbulence modeling via the fractional Laplacian}, arXiv preprint arXiv:1803.05286, 2018.

\bibitem{fbm}
A. Fiscella , G. M. Bisci and R. Servadei, \emph{Bifurcation and multiplicity results for critical nonlocal fractional Laplacian problems}, Bulletin des Sciences Mathématiques 140 (2016), 14-35.


\bibitem{FiscellaVecchi} 
A. Fiscella and E. Vecchi, \emph{Bifurcation and multiplicity results for critical magnetic fractional problems}, Electronic Journal of Differential Equations 153 (2018), 1-18.

\bibitem{FiscellaV1}
A. Fiscella, A. Pinamonti and E. Vecchi,
{\it Multiplicity results for magnetic fractional problems},
J. Differential Equations {\bf 263} (2017), 4617--4633. 


\bibitem{JPS}
J. Giacomoni, P. K. Mishra, and K. Sreenadh, \emph{Fractional elliptic equations with critical exponential nonlinearity}, 
Adv. Nonlinear Anal.  5 (2016), 57-74.

\bibitem{IzSq}
A. Iannizzotto and  M. Squassina, \emph{1/2-Laplacian problems with exponential nonlinearity}, J. Math. Anal. Appl.  414 (2014), 372-385.
 \bibitem{T1} T. Ichinose, \emph{Essential selfadjointness of the Weyl quantized relativistic Hamiltonian}, 
Ann. Inst. H. Poincar\'e Phys. Th\'eor.,   51 (1989), 265-–297.
\bibitem{T2}T. Ichinose and H. Tamura, \emph{Imaginary-time path integral for a relativistic spinless particle
in an electromagnetic field}, Commun. Math. Phys.,  105 (1986), 239-–257.


\bibitem{matrinazziint} 
S. Iula, A. Maalaoui and L. Martinazzi, \emph{A fractional Moser-Trudinger type inequality in one dimension and its critical points}, Differential Integral Equations 29 (2016), 455-492.

\bibitem{Wadade}
H. Kozono,  T. Sato and H. Wadade, \emph{Upper bound of the best
	constant of a Trudinger-Moser inequality and its application to a
	Gagliardo-Nirenberg inequality}. Indiana Univ. Math. J.
55 (2006), 1951-1974.


\bibitem{Martinazzi}
L. Martinazzi, \emph{Fractional Adams-Moser-Trudinger type inequalities}, Nonlinear Anal. 127 (2015), 263–278.

\bibitem{EJDE2} 
P. K. Mishra and K. Sreenadh, \emph{Bifurcation and multiplicity of solutions for the fractional Laplacian with critical exponential nonlinearity},
Electronic Journal of Differential Equations 203 (2016), 1-9.


\bibitem{moser}
J. Moser, \emph{A sharp form of an inequality by N. Trudinger},
Indiana Univ. Math. J. 20 (1970/71), 1077-1092.


\bibitem{Ozawa}
T. Ozawa, \emph{On critical cases of Sobolev's inequalities}, J.
Funct. Anal. 127 (1995), 259-269.

\bibitem{AME1}
A. Pinamonti, M. Squassina and E. Vecchi,
{\it Magnetic BV functions and the Bourgain-Brezis-Mironescu formula}, 
to appear on Advances in Calculus of Variations, Preprint. arXiv:1609.09714. 


\bibitem{AME}
A. Pinamonti, M. Squassina and E. Vecchi,
{\it The Maz'ya-Shaposhnikova limit in the magnetic setting},
J. Math. Anal. Appl. {\bf 449} (2017), 1152--1159.

\bibitem{Ricceri} 
B. Ricceri, \emph{On a classical existence theorem for nonlinear elliptic equations, in: M. Thera (Ed.)}, Experimental, Constructive and Nonlinear Analysis, in: Conf. Proc., Can. Math. Soc., 27 (2000), 275-278.


\bibitem{Takahasi}
F. Takahashi, \emph{Critical and subcritical fractional Trudinger-Moser-type inequalities on $\mathbb{R}$},
Adv. Nonlinear Anal. 8 (2019), 868-884.

 \bibitem{FL4} J. Tan, \emph{The Brezis-Nirenberg type problem involving the square root of the Laplacian},  Calc. Var. Partial Differential Equations, 36 (2011), 21-41.	

\bibitem{BMX}
B. Zhang, M. Squassina and X. Zhang, 
{\it Fractional NLS equations with magnetic field, critical frequency and critical growth},
Manuscripta Math. {\bf 155} (2018), no. 1-2, 115--140. 


\end{thebibliography}
 \end{document}